\documentclass[11pt]{article}
\usepackage{a4wide,amsmath,amssymb,latexsym,amsthm,bbm,color,enumitem,xargs,ifthen}
\usepackage[utf8]{inputenc}
\usepackage[authoryear]{natbib}
\usepackage{rotating}
\usepackage{authblk}

\RequirePackage[colorlinks,linkcolor=red, urlcolor=blue,citecolor=blue]{hyperref}

\def\Rset{\mathbb R}
\def\nset{\mathbb N}
\def\zset{\mathbb Z}
\def\cset{\mathbb C}
\def\rset{\mathbb R}
\def\cov{\mathrm{Cov}}
\def\esp{\mathbb E}
\def\pr{\mathbb P}
\def\var{\mathrm{Var}}
\def\1{\mathbbm 1}
\def\rmd{\mathrm d}
\def\rme{\mathrm e}
\def\rmi{\mathrm i}
\def\rmI{\mathrm I}  
\def\zed{{z}}

\newcommandx{\stat}[2][2=]{{#1}^{<\text{\tiny{S}}>\,#2}}
\newcommand{\locstat}[1]{{#1}^{<\text{\tiny{LS}}>}}
\DeclareMathOperator*{\esssup}{ess\,sup}

\newcommandx{\betnorm}[3][2=\beta,3=\delta]{\left|#1\right|_{\langle#2,#3\rangle}}
\newcommandx{\betonenorm}[2][2=\beta]{\left|#1\right|_{\{#2\}}}
\newcommandx{\ellnorm}[2][2=1]{\left|#1\right|_{#2}}
\newcommandx{\ellexpnorm}[3][2=1,3=d]{\left|#1\right|_{#3,#2}}
\newcommandx{\ellbetanorm}[3][2=1,3=\beta]{\left|#1\right|_{(#3),#2}}

\def\mcesy{\mathcal{E}}
\def\mcesyt{\tilde{\mathcal{E}}}
\newcommandx{\mce}[1][1=g]{\ifthenelse{\equal{#1}{g}}{\mcesy}{\mcesyt}}
\newcommandx{\fixedpoint}[1][1=g]{\ifthenelse{\equal{#1}{g}}{g}{\tilde{g}}}
\newcommandx{\mcei}[1][1=g]{\fixedpoint[#1]_{\infty}}
\newcommandx{\mcet}[1][1=g]{\mce[#1]^{(T)}}
\newcommandx{\mceti}[1][1=g]{\fixedpoint[#1]_{\infty}^{(T)}}
\newcommandx{\mces}[1][1=g]{\stat{\mce[#1]}}
\newcommandx{\mcesi}[1][1=g]{\stat{\fixedpoint[#1]}_{\infty}}

\newcommandx{\ballellnorm}[2][2=1]{B_{#2}\left(#1\right)}
\newcommandx{\ballellnormp}[2][2=1]{B_{+,#2}\left(#1\right)}

\newcommandx{\D}[1][1=\ouvert]{\bar{\mathcal{O}}\left(#1\right)}
\newcommandx{\partialD}{\partial_{\mathcal{O}}}
\newcommandx{\Dp}[2][1=1,2=\ouvert]{\bar{\mathcal{O}}_{#1}\left(#2\right)}

\newcommandx{\normD}[3][2=1,3=K]{\left|#1\right|_{\bar{\mathcal{O}},#3,#2}}
\newcommandx{\bdp}[3][2=1,3=K]{B_{\bar{\mathcal{O}}}\left(#1;#3,#2\right)}
\newcommandx{\bdbet}[3][2=\beta,3=K]{B_{\bar{\mathcal{O}}}\left(#1;#3,(#2)\right)}
\newcommandx{\betseminormD}[3][2=\beta,3=K]{\left|#1\right|_{\bar{\mathcal{O}},#3,(#2)}}

\newcommandx{\dd}[1][1=\ouvert]{\mathcal{O}\left(#1\right)}
\newcommandx{\normd}[2][2=K]{\left|#1\right|_{\mathcal{O},#2}}

\newcommandx{\ballbetnorm}[2][2=\beta]{B_{\langle#2\rangle}\left(#1\right)}
\newcommandx{\allnorm}[2][2=\beta]{\left|#1\right|_{[#2]}}
\newcommandx{\ballallnorm}[2][2=\beta]{B_{[#2]}\left(#1\right)}
\newcommandx{\ballallnormp}[2][2=\beta]{B_{+,[#2]}\left(#1\right)}

\newcommandx{\zzeta}[1][1=1]{{\zeta}_{#1}}
\newcommandx{\zetals}[1][1=1]{\locstat{\zeta}_{#1}}
\newcommandx{\zetabet}[1][1=\beta]{\locstat{\zeta}_{(#1)}}
\newcommandx{\holderls}[1][1=\beta]{\xi^{(#1)}}
\newcommandx{\holderlsc}[1][1=\beta]{\xi_c^{(#1)}}
\newcommandx{\zetaexp}[2][1=1]{{\zeta}_{#1}(#2)}
\newcommandx{\zetaexpls}[2][1=1]{\locstat{\zeta}_{#1}(#2)}

\newcommand{\beq}{\begin{equation}}
\newcommand{\eeq}{\end{equation}}

\def\ouvert{U}

\def\supp{\mathrm{Supp}}

\theoremstyle{definition}

\newtheorem{convention}{Convention}
\theoremstyle{plain}
\newtheorem{theorem}{Theorem}
\newtheorem{proposition}[theorem]{Proposition}
\newtheorem{corollary}[theorem]{Corollary}
\newtheorem{lemma}[theorem]{Lemma}
\newtheorem{remark}{Remark}

\newenvironment{assumption}[1]{\begin{enumerate}[label=(\textbf{\sf #1}-\arabic*),resume=hyp#1]\begin{sf}}{\end{sf}\end{enumerate}}

\title{Time-frequency analysis of locally stationary Hawkes processes}
\author[1]{F.\,Roueff}
\author[2]{R.\,von Sachs}
\date{\today}
\affil[1]{LTCI, Télécom ParisTech, Université Paris-Saclay}
\affil[2]{Institut de statistique, biostatistique et sciences actuarielles (ISBA) IMMAQ, Université catholique de Louvain}
\newcommandx{\bfreq}{\ensuremath{b_2}}
\newcommandx{\btime}{\ensuremath{b_1}}
\newcommandx{\ktime}{\ensuremath{w}}
\newcommandx{\kbtime}{\ensuremath{w_{\btime}}}
\newcommandx{\ktbtime}{\ensuremath{w_{T\btime}}}
\newcommandx{\fktime}{\ensuremath{W}}
\newcommandx{\fkbtime}{\ensuremath{W_{\btime}}}
\newcommandx{\fktbtime}{\ensuremath{W_{T\btime}}}
\newcommandx{\kfreq}{\ensuremath{q}}
\newcommandx{\kbfreq}{\ensuremath{q}_{\bfreq}}
\newcommandx{\fkfreq}{\ensuremath{Q}}
\newcommandx{\kbofreq}[1][1=\omega_0]{\ensuremath{q_{#1,\bfreq}}}
\newcommandx{\fkbofreq}[1][1=\omega_0]{\ensuremath{Q_{#1,\bfreq}}}

\newcommandx{\fourierp}{\ensuremath{P}}
\newcommandx{\fourierf}{\ensuremath{F}}
\newcommand{\conjugate}[1]{{#1}^{*}}

\begin{document}

\maketitle
\abstract{Locally stationary Hawkes processes have been introduced in order to
  generalise classical Hawkes processes away from stationarity by allowing for
  a time-varying second-order structure. This class of self-exciting point
  processes has recently attracted a lot of interest in applications in the
  life sciences (seismology, genomics, neuro-science,...), but also in the
  modeling of high-frequency financial data. In this contribution we provide a
  fully developed nonparametric estimation theory of both local mean density
  and local Bartlett spectra of a locally stationary Hawkes process. In
  particular we apply our kernel estimation of the spectrum localised both in
  time and frequency to two data sets of transaction times revealing pertinent
  features in the data that had not been made visible by classical
  non-localised approaches based on models with constant fertility functions
  over time.}

\medskip



\section{Introduction}
\label{sec:intro}

Many recent time series data modelling and analysis problems increasingly face
the challenge of occurring time-variations of the underlying probabilistic
structure (mean, variance-covariance, spectral structure,....). This is due to
the availability of larger and larger data stretches which can hardly any
longer be described by stationary models. Mathematical statisticians
(\cite{dahlhaus96}, \cite{zhou-wu2009}, \cite{hallin-et-al2014}, among others)
have contributed with time-localised estimation approaches for many of these
time series data based on rigorous models of {\em locally stationary
  approximations} to the non-stationary data. Often it has been only via these
theoretical studies that well-motivated and fully understood {\em
  time-dependent estimation} methods could be developed which correctly adapt
to the degree of deviation of the underlying data from a stationary situation
(e.g. via the modelling of either a slow - or in contrast a rather abrupt --
change of the probabilistic structure over time). For many of these situations,
the development of an asymptotic theory of doubly-indexed stochastic processes
has proven to be useful: the underlying data stretch is considered to be part
of a family of processes which asymptotically approaches a limiting process
which locally shows all the characteristics of a stationary process (hence
``the locally stationary approximation''). Accompanying estimators ought to
adapt to this behavior, e.g. by introduction of local bandwidths in time.  Many
of these aforementioned approaches have been achieved in the context of
classical real-valued series in discrete time, e.g. for (linear) time series
via time-varying MA($\infty$) representations which apply for a large class of
models (see, e.g., \cite{dahlhaus2000}).
Once those are available, the development of a rigorous asymptotic estimation
theory is achievable.

This approach, however, is not directly applicable to the model of {\em locally
  stationary Hawkes processes}, a class of self-exciting point processes
introduced in \cite{roueff-vonsachs-sansonnet2016} with not only time-varying
baseline intensity (such as in \cite{chen-hall-2013}) but also time-varying
fertility function. The fertility function describes the iterative
probabilistic mechanism for Hawkes processes to generate offsprings from each
occurrence governed by a conditional Poisson point process, given all previous
generations. Often, a fertility function $p(t)$ with exponential decay over
time $t$ is assumed. The condition $\int p<1$ ensures a non-explosive
accumulation of consecutive populations of the underlying process. When this
fertility function is made varying over time through a second argument, one can
still rely on the point process mechanism to derive locally stationary
approximations.

In this paper we treat estimation of the first and second order structure
of locally stationary Hawkes processes on the real line, with a
(time-dependent) fertility function $p(\cdot; t)$ assumed to be causal,
i.e. supported on $\rset_+$.  Whereas for estimating the local mean density of
the process it is sufficient to introduce a localisation via a short window
(kernel) in time, for estimating the second-order structure, i.e. the local
Bartlett spectrum of the process, one needs to localise both in time and
frequency: this time-frequency analysis will be provided via a pair of kernels
which concentrate around a given point $(t,\omega)$ in the time-frequency
plane. Using an appropriate choice of bandwidths in time and in frequency which
tend to zero with rates calibrated to minimize the asymptotic mean-square error
between the time-frequency estimator and the true underlying local Bartlett
spectrum, one can show consisteny of our estimator of the latter one.  Note
that the use of kernels for non-parametric estimation with counting process is
not new. To the best of our knowledge this was introduced first by
\cite{ramlau-hansen83} for estimating regression point process models such as
those introduced in \cite{aalen75}, see also \cite{andersen85} for a general
account on the estimation of such processes.

Self exciting point processes have been recently used for modelling point
processes resulting from high frequency financial data such as price jump
instants (see e.g. \cite{bacry-delattre-hoffmann-muzy-2010}) or limit order
book events (see e.g. \cite{zheng-roueff-abergel2014}). In this paper we will
illustrate our time frequency analysis approach for point processes on
transaction times of two assets (which are Essilor International SA and Total
SA).

The rest of the paper is organised as follows. In
Section~\ref{sec:main-defin-assumpt}, we introduce all necessary preparatory
material to develop our estimation theory, including the definition of the
population quantities, i.e. local mean density and local Bartlett spectra that
we wish to estimate. Section~\ref{sec:asymp_est_theory} defines our kernel
estimators and treats asymptotic bias and variance developments of those under
regularity assumptions to be given beforehand. Although these developments are
far from being direct and straightforward, the resulting rates of convergence
are completely intuitive from a usual nonparametric time-frequency estimation
point of view.  Section~\ref{sec:numer-exper} presents the analysis of our
transaction data sets leading to interesting observations that have not
previously been revealed by a classical analysis with time-constant fertility
functions. In Section~\ref{sec:dev-bounds}, we present, as a result of
independent interest in its own, the necessary new techniques of directly
controlling moments of non-stationary Hawkes processes. 
All proofs
are deferred to the Appendix.

\section{Preparatory material}\label{sec:main-defin-assumpt}

In this preparatory section we prepare the ground for developing the
presentation of our asymptotic estimation theory of local Bartlett spectra. We
do this by introducing a list of useful conventions and definitions, as well as
recalling the main concepts and results of \cite{roueff-vonsachs-sansonnet2016}
in as much as they are necessary.
\subsection{Conventions and notation}
\label{sec:conventions-notation}
We here set some general conventions and notation adopted all along the paper.
Additional ones are introduced in relation with the main assumptions in
Section~\ref{sec:LS-Assumptions}. 

\noindent A point process is identified with a random measure with discrete support,
$N=\sum_k \delta_{t_k}$ typically, where $\delta_t$ is the Dirac measure at point
$t$ and $\{t_k\}$ the corresponding (locally finite) random set of points. We use
the notation $\mu(g)$ for a measure $\mu$ and a function $g$ to express $\int g
\, \rmd \mu$ when convenient. In particular, for a measurable set $A$,
$\mu(A)=\mu(\1_A)$ and for a point process $N$, $N(g)=\sum_k g(t_k)$.
The shift operator of lag $t$ is denoted by  $S^{t}$. For a set $A$,
$S^{t}(A)=\{x-t, \, x \in A\}$ and for a function $g$, $S^{t}(g)=g(\cdot+t)$,
so that $S^{t}(\1_A)=\1_{S^{t}(A)}$. One can then compose a measure $\mu$ with
$S^t$, yielding for a function $g$, $\mu\circ S^t(g)=\mu(g(\cdot+t)$.

We also need some notation for the functional norms which we deal with in this work.
Usual $L^p$-norms are denoted by $\ellnorm{h}[p]$,
$$
\ellnorm{h}[p]=\left(\int |h|^p\right)^{1/p} \;,
$$
for $p\in[1,\infty)$ and $\ellnorm{h}[\infty]$ is the essential supremum on
$\rset$,
$\ellnorm{h}[\infty]=\esssup_{t\in\rset}\left|h(t)\right| \;.$
We also use the following  weighted  $L^p$ norms which we define to be for any $p\geq1$, 
$\beta>0$, $a\geq0$ and $h:\rset\to\cset$,
\begin{align}
  \label{eq:pow-weight-norm}
&\ellbetanorm{h}[p][\beta]
:=\ellnorm{h\times|\cdot|^\beta}[p]=\left(\int
  \left|h(t)\ t^\beta\right|^p\rmd t\right)^{1/p}\; ,\\
  \label{eq:exp-weight-norm}
&\ellexpnorm{h}[p][a]:=
\ellnorm{h\times\rme^{a|\cdot|}}[p]=\left(\int |h(t)|^p\,\rme^{a\,p\,|t|}\;\rmd
  t\right)^{1/p}\; ,
\end{align}
with the above usual essential sup extensions to the case $p=\infty$.

\noindent We
denote the convolution product by $\ast$, that is,
for any two functions $h_1$ and $h_2$,
$$
h_1\ast h_2(s)=\int h_1(s-t)h_2(t)\rmd t\;.
$$
Finally we use for a random variable $X$ the notation 
\begin{equation}
  \label{eq:notation-q-norm-Lq}
 \left\|X\right\|_p :=
\left(\esp |X|^p\right)^{1/p}, \quad\text{for}\quad  p \geq1\;.  
\end{equation}

\subsection{From stationary to non-stationary and locally stationary Hawkes processes}
\label{sec:hawkes-processes-as-cluster}

To start with, we first recall the definition of a stationary (linear) Hawkes
process $N$ with immigrant intensity $\lambda_c$ and fertility function $p$ defined
on the positive half-line. The conditional intensity function $\lambda(t)$ of
such a process is driven by the fertility function taken at the time distances
to previous points of the process, i.e. $\lambda(t)$ is given by
\begin{equation}\label{eqn:motiv}
\lambda(t)= \lambda_c + \int_{-\infty}^{t^-} p(t-s)\;N(\rmd s)\ = \ \lambda_c + \sum_{t_i < t} p(t-t_i) \; .
\end{equation}
Here the first integral is to be read as the integral of the ``fertility''
function $p$ with respect to the counting process $N$, which is a sum of Dirac
masses at (random) points $(t_i)_{i\in\zset}$.  The existence of a stationary
point processes with conditional intensity~(\ref{eqn:motiv}) holds under the
condition $\int p<1$ and  can be constructed as a cluster point process (see
\cite[\textsc{Example}~6.3(c)]{daley-vere-jones-2003}).

We extend the stationary Hawkes model defined by the conditional
intensity~(\ref{eqn:motiv}) to the non-stationary case by authorizing the
immigrant intensity $\lambda_c$ to be a function $\lambda_c(t)$ of time $t$ and also the
fertility function $p$ to be time varying, replacing $p(t-s)$ by the more
general $p(t-s;t)$. To ensure a locally finite point process in this
definition, we impose the two conditions
\begin{equation}
  \label{eq:cond-density-nonstat}
\zzeta:=\sup_{t\in\rset}\int p(s;t)\,\rmd s< 1\quad\text{and}\quad \ellnorm{\lambda_c}[\infty]<\infty\;.
\end{equation}
They yield the existence of a non-stationary point process $N$ with a mean density function which is uniformly bounded
by $\ellnorm{\lambda_c}[\infty]/(1-\zzeta)$ (see
\cite[Definition~1]{roueff-vonsachs-sansonnet2016}).  

As non-stationary Hawkes processes can evolve quite arbitrarily over time, the
statistical analysis of them requires to introduce local stationary
approximations in the same fashion as time varying autoregressive processes in
time series, for which \emph{locally stationary} models have been successfully
introduced (see~\cite{dahlhaus96}).  Thus, a locally stationary Hawkes process
with \emph{local immigrant intensity} $\locstat{\lambda}_c(\cdot)$ and
\emph{local fertility function} $\locstat{p}(\cdot;\cdot)$ is a collection
$(N_{T})_{T\geq1}$ of non-stationary Hawkes processes with respective immigrant
intensity and fertility function given by
$\lambda_{cT}(t)=\locstat{\lambda}_c(t/T)$ and varying fertility function given
by $p_T(\cdot;t)=\locstat{p}(\cdot;t/T)$, see
\cite[Definition~2]{roueff-vonsachs-sansonnet2016} where this model is called a
\emph{locally stationary} Hawkes process. For a given \emph{real location} $t$,
the scaled location $t/T$ is typically called an \emph{absolute location} and
denoted by $u$ or $v$.

Note that the collection
$(N_{T})_{T\geq1}$ of non-stationary Hawkes processes are defined using the same
time varying parameters  $\locstat{\lambda}_c$ and
$\locstat{p}$ but with the time varying arguments scaled by $T$.
As a result, the larger $T$ is, the slower the parameters evolve along the
time.

An assumption corresponding to~(\ref{eq:cond-density-nonstat}) to guarantee
that, for all $T\geq1$, the non-stationary Hawkes process $N_{T}$ admits
a uniformly bounded mean density function is the following:
\begin{equation} \label{eq:locally-stat-hawk-finite-intensity}
\zetals:=\sup_{u\in\Rset} \int \locstat{p}(r;u)\; \rmd r < 1
\quad\text{and}\quad\ellnorm{\locstat{\lambda}_c}[\infty]<\infty \; .
\end{equation}
Under this assumption, for each absolute location $u\in\rset$, the function
$r\mapsto\locstat{p}(r;u)$ satisfies the required condition for the fertility
function of a stationary Hawkes process. Hence, assuming
(\ref{eq:locally-stat-hawk-finite-intensity}), for any absolute location $u$,
we denote by $N(\cdot;u)$ a stationary Hawkes process with constant immigrant
intensity $\locstat{\lambda}_c(u)$ and fertility function
$r\mapsto\locstat{p}(r;u)$. In the following subsection we will include this
assumption~(\ref{eq:locally-stat-hawk-finite-intensity}) into a stronger set of
assumptions that we use for derivation of the results on asymptotic estimation
theory.

We also remark that, for any $T\geq1$, the conditional intensity function
$\lambda_T$ of the non-stationary Hawkes process $N_T$ takes the form
\begin{align*} 
\lambda_T(t)& = \locstat{\lambda}_c\left(t/T\right) + \int_{-\infty}^{t^-}
\locstat{p}\left(t-s;t/T\right)\;N_T(d s)\\
& =\locstat{\lambda}_c(t/T)+\sum_{t_{i,T}<t} \locstat{p}(t-t_{i,T};t/T)\;,
\end{align*}
where $(t_{i,T})_{i\in\zset}$ denote the points of $N_T$.
This latter formula can also be used to simulate locally stationary
Hawkes processes on the real line. 
The examples for locally stationary Hawkes processes (with time varying Gamma shaped
fertility functions) used in~\cite[Section
2.6]{roueff-vonsachs-sansonnet2016} were simulated in this way.

First and second order statistics for point processes are of primary 
importance for statistical inference. As for time series they are conveniently
described in the stationary case by a mean parameter for the first order
statistics and a spectral representation, the so called Bartlett spectrum (see
\cite[Proposition~8.2.I]{daley-vere-jones-2003}), for the covariance
structure. The locally stationary approach allows us to define such quantities
as depending on the absolute time $u$ as introduced in the following section.

\subsection{Local mean density and Bartlett spectrum}
\label{sec:local-mean-density-bartlett}
Consider a locally stationary Hawkes process $(N_{T})_{T\geq1}$ with local
immigrant intensity $\locstat{\lambda}_c$ and local fertility function
$\locstat{p}(\cdot;\cdot)$ satisfying
condition~(\ref{eq:locally-stat-hawk-finite-intensity}).  Although for a
given $T$, the first and second order statistics of $N_T$ can be quite
involved, some intuitive asymptotic approximations are available as $T$ grows
to infinity.  Namely, for any absolute time $u$, the \emph{local} statistical
behavior of $N_T$ around real time $Tu$ has to be well approximated by that of a
stationary Hawkes process with (constant) immigrant intensity
$\locstat{\lambda}_c(u)$ and fertility function
$\locstat{p}(\cdot;u)$.  This stationary Hawkes process at
absolute location $u$ is denoted in the following by $N(\cdot;u)$.  Precise
approximation results are provided in \cite{roueff-vonsachs-sansonnet2016} and
recalled in Section~\ref{sec:appr-results-locally} below. Presently, we only
need to introduce how to define this local first and second order statistical
structure.

We first introduce the \emph{local mean density} function $\locstat{m}_1(u)$ defined at each
absolute location $u$, as the mean
parameter of the stationary Hawkes process  $N(\cdot;u)$. By \cite[Eq.~(6.3.26)
in Example~6.3(c)]{daley-vere-jones-2003}, it is given by
\begin{equation}
  \label{eq:local-mean}
\locstat{m}_1(u)=\frac{\locstat{\lambda}_c(u)}{1-\int \locstat{p}(\cdot;u)}\;.
\end{equation}  
A convenient way to describe the covariance structure of a second order
stationary point process $N$ on $\rset$ is to rely on a spectral
representation, the Bartlett spectrum, which is defined as the (unique)
non-negative measure $\Gamma$ on the Borel sets such that, for any bounded and
compactly supported function $f$ on $\rset$, (see \cite[Proposition~8.2.I]{daley-vere-jones-2003})
$$
  \mathrm{Var}\big(N(f)\big)=\Gamma(|\fourierf|^2)=\int \left|\fourierf(\omega)\right|^2\;\Gamma(\rmd\omega)\;,
$$
where $\fourierf$ denotes the Fourier transform of $f$,
$$
\fourierf(\omega)=\int f(t)\;\rme^{-\rmi t \omega}\;\rmd t \;.
$$
For the stationary Hawkes processes $N(\cdot;u)$, the Bartlett spectrum admits a density
given by (see \cite[Example~8.2(e)]{daley-vere-jones-2003})
\begin{align}
 \locstat{\gamma}(u;\omega)
&=\ \frac{\locstat{m}_1(u)}{2\pi}\,\left|1-\locstat{\fourierp}(\omega;u)\ \right|^{-2}\, ,
  \label{eq:local-bartlett-spectral-density}
\end{align}
where
$$
\locstat{\fourierp}(\omega;u)= \int \locstat{p}(t;u)\;\rme^{-\rmi t \omega}\;\rmd t \;.
$$
Analagous to the first order structure, we call $\locstat{\gamma}(u;\omega)$ the \emph{local
  Bartlett spectrum density} at frequency $\omega$ and absolute location $u$. 
This local Bartlett spectrum density plays a role similar to that of the local
spectral density $f(u,\lambda)$ introduced in \cite[Page~142]{dahlhaus96} for
locally stationary time series.

\subsection{Estimators}
\label{sec:moment-estimation}

As our approach is local in time and frequency, we rely on two kernels $\ktime$
and $\kfreq$ which are required to be compactly supported 
(see Remark~\ref{rem:finite-set-of
  observations} below). More precisely, we have the following assumptions.
\begin{assumption}{K}
\item \label{ass:ktime-comp-support}
Let  $\ktime$ be  a $\rset\to\rset_+$ bounded function with compact support
such that $\int \ktime = \ellnorm{\ktime} =1$.
\item \label{ass:kfreq-comp-support} Let $\kfreq$  be  a $\rset\to\cset$ bounded function with compact support
  such that $\ellnorm{\kfreq}[2] = \sqrt{2\pi}$.
\end{assumption}
To localize in time let $\btime >0$ be a given time bandwidth and define
$\kbtime$ and $\ktbtime$ to be the corresponding kernels in absolute time $u$
and real time $t$, namely,
\begin{equation}
  \label{eq:ktbtime-def}
\kbtime(u):= \btime^{-1}\,\ktime(u/\btime)\quad\text{and}\quad
\ktbtime(t):= T^{-1}\,\kbtime(t/T)= (T\btime)^{-1}\,\ktime(t/(T\btime)) \;. 
\end{equation} 
Let now $u_0$ be a fixed absolute time.
For estimating the local mean density $\locstat{m}_1(u_0)$ given by
equation~(\ref{eq:local-mean}), approximating the mean density function
$t\mapsto m_{1T}(t)$ of $N_T$ locally in the neighborhood of $T u_0$ by
$\locstat{m}_1(u_0)$ we have, for $\btime$ small,
$$
\locstat{m}_1(u_0)\approx \int
\ktbtime(t-Tu_0)m_{1T}(t)\,\rmd t\;,
$$
where we used $\int\ktbtime=1$ and that the support of
$t\mapsto\ktbtime(t-Tu_0)$ essentially lives in the neighborhood of $T u_0$
for $T\btime$ small. Since the right-hand side of this approximation is
$\esp\left[N_T(S^{-Tu_0}\ktbtime)\right]$, this suggests the following
estimator of $\locstat{m}_1(u_0)$,
\begin{equation}\label{eq:intensity_est}
\widehat{m}_{\btime}(u_0)\ :=  \ N_T(S^{-Tu_0}\ktbtime)\ = \int
\ktbtime(t-Tu_0)\ N_T(\rmd t)\ .
\end{equation}

For estimation of the second order structure, i.e. the local Bartlett spectral
density $\locstat{\gamma}(u_0;\omega_0)$ for some given point
$(u_0, \omega_0)$ of the time-frequency plane, we need also to localise in
frequency by a kernel which will be given by the (squared) Fourier transform
$|\fkfreq|^2$ of the kernel $\kfreq$. Then for a given frequency bandwidth
$\bfreq >0$, we are looking for an estimator of the auxiliary quantity
\begin{equation}
  \label{eq:regularized}
  \locstat{\gamma}_{\bfreq}(u_0;\omega_0):=\int \frac{1}{\bfreq}
\left|\fkfreq\left(\frac{\omega-\omega_0}{\bfreq}\right)\right|^2\;\locstat{\gamma}(\omega;u_0)\rmd\omega,
\end{equation}
which in turn, as $\bfreq\to0$, is an approximation of the density
$\locstat{\gamma}(u_0;\omega_0)$, since~\ref{ass:kfreq-comp-support} implies
$\ellnorm{\fkfreq}[2] = 1$.  The rate of approximation (i.e. the
``bias in frequency direction'' of the following estimator) is established in
Theorem~\ref{thm:total-bias}, equation~(\ref{eq:freq_bias}), below.  Let us now
set $\kbofreq(t)=\bfreq^{1/2} \rme^{\rmi\omega_0t}\kfreq(\bfreq t)$ such that
the squared modulus of its Fourier transform writes as
$$
|\fkbofreq(\omega)|^2=
\frac{1}{\bfreq} \left|\fkfreq(\frac{\omega-\omega_0}{\bfreq})\right|^2\;.
$$
Using that $\locstat{\gamma}(\omega;u_0)\rmd\omega$ is the Bartlett spectrum of
$N(\cdot;u_0)$ as recalled in Section~\ref{sec:local-mean-density-bartlett}, we
can thus rewrite~(\ref{eq:regularized}) as 
\begin{equation}
  \label{eq:regularized-new}
  \locstat{\gamma}_{\bfreq}(u_0;\omega_0)=
\var\left(N(\kbofreq;u_0)\right) \;.
\end{equation}
Since this variance is an approximation of
$\var\left(N_T(S^{-Tu_0}\kbofreq)\right)$, where $\kbofreq$ is shifted to be localized around $T\,u_0$,
we finally estimate $\locstat{\gamma}_{\bfreq}(u_0;\omega_0)$ by the
following moment estimator:
\begin{equation}
  \label{eq:kernel-est}
\widehat{\gamma}_{\bfreq,\btime}(u_0;\omega_0) :=
\widehat{E}\left[|N_T(S^{-Tu_0}\kbofreq)|^2;\kbtime\right]
-\left|\widehat{E}\left[N_T(S^{-Tu_0}\kbofreq);\kbtime\right] \right|^2 ,
\end{equation}
where for the test function $f=S^{-Tu_0}\kbofreq$ and taking $\rho(x)=x$ and $\rho(x)=|x|^2$ successively,
we have built estimators of $\esp[\rho(N_T(f))]$ based on the empirical
observations of $N_T$ and defined by
\begin{align}
  \label{eq:general_moment_est}
\widehat{E}[\rho\left( N_T(f)\right);\kbtime] :=
 \int \rho\left( N_T(f(\cdot-t))\right)\ \ktbtime(t)\;\rmd t\\
\nonumber
=
\frac1T \int \rho\left( \sum_k f(t_{k,T}-t)\right)\ \kbtime(t/T)\;\rmd t\;.
\end{align}
Note that in~(\ref{eq:kernel-est}) the dependence of the estimator on
$u_0,\omega_0$ appears in the choice of $f=S^{-Tu_0}\kbofreq$.
By an obvious change of variable, this would be
equivalent to let the kernel $\kbofreq$ unshifted in time, hence take $f=\kbofreq$, and instead shift
$\ktbtime(t)$ into $\ktbtime(t-Tu_0)$, or, in absolute time, shift $\kbtime(u)$
into $\kbtime(u-u_0)$.
\begin{remark}\label{rem:finite-set-of observations}
  In practice $N_T$ is observed over a finite interval. In order to
  have estimators $\widehat{m}_{\btime}(u_0)$ and
  $\widehat{\gamma}_{\bfreq,\btime}(u_0;\omega_0)$ in~(\ref{eq:intensity_est})
  and~(\ref{eq:kernel-est}) that only use observations within this interval,
  the supports of $\ktime$ and $\kfreq$ must be bounded and some restriction
  imposed on $\btime$, $\bfreq$ and $T$. Suppose for instance that $N_T$ is
  observed on $[0,T]$ (mimicking the usual convention for locally stationary
  time series of \cite{dahlhaus95}). The local mean density and Bartlett
  spectrum can then be estimated at corresponding absolute times $u_0\in(0,1)$ and
  the restrictions on $\btime$, $\bfreq$ and $T$ read as follows.
  In~(\ref{eq:intensity_est}), we must have
$u_0+\btime\supp(\ktime)\subseteq[0,1]$,
and in~(\ref{eq:kernel-est}), we must have
$u_0+\btime\supp(\ktime)+(T\bfreq)^{-1}\supp(\kfreq)\subseteq[0,1]$.
These two support conditions are always satisfied, eventually as $T\to\infty$,
provided that the kernels $\ktime$ and $\kfreq$ are compactly supported and
that $\btime\to0$ and $T\bfreq\to\infty$.
\end{remark}
In the sequel we will show that this is a sensible estimator of
$\locstat{\gamma}_{\bfreq}(u_0;\omega_0)$ sharing the usual properties of a nonparametric
estimator constructed via kernel-smoothing over time and frequency: for
sufficiently small bandwidths $\btime$ in time and $\bfreq$ in frequency this
estimator becomes well localised around $(u_0;\omega_0)$. 

The main results stated hereafter provide asymptotic expansions
of its bias and variance behaviour, 
leading to consistency of this estimator under some asymptotic condition for
$\btime$ and $\bfreq$ as $T\to\infty$.
\section{Bias and variance bounds}
\label{sec:asymp_est_theory}

\subsection{Main assumptions}
\label{sec:LS-Assumptions}

The first assumption is akin but stronger than
condition~(\ref{eq:locally-stat-hawk-finite-intensity}) above, being in fact
equal to assumption (LS-1) of \cite{roueff-vonsachs-sansonnet2016}. It guarantees
that, for all $T\geq1$, the locally stationary Hawkes process $N_{T}$ admits a
(causal) local fertility function $s\mapsto\locstat{p}(s;u)$ which is not only
uniformly bounded, but has an exponentially decaying memory (as a function in
the first argument, uniformly with respect to its second argument).
\begin{assumption}{LS}
\item \label{ass:locally-stat-hawk-finite-expbounds}
Assume that 
\begin{equation}
\ellnorm{\locstat{\lambda}_c}[\infty]<\infty \; .
\end{equation}
Assume moreover that for all $u\in\Rset$, $\ \locstat{p}(\cdot;u)\ $ is supported on $\rset_+$ and that
there exists a $d>0$ such that $\zetaexpls{d}<1$ and $\zetaexpls[\infty]{d}<\infty$ where 
\begin{equation}\label{eq:exp-decay-l1}
\zetaexpls{d}:=\sup_{u\in\Rset}\ellexpnorm{\locstat{p}(\cdot;u)}[1][d]
=\sup_{u\in\Rset} \int \locstat{p}(s;u)\; \rme^{d|s|} \rmd s
\end{equation}
and
\begin{equation}\label{eq:exp-decay-linfty}
  \zetaexpls[\infty]{d}:=\sup_{u\in\Rset}\ellexpnorm{\locstat{p}(\cdot;u)}[\infty][d]
  =\sup_{u\in\Rset} \esssup_{s\in\Rset}\left\{|\locstat{p}(s;u)|\rme^{d|s|}\right\}
\;.
\end{equation}
\end{assumption}
All the examples considered in \cite{roueff-vonsachs-sansonnet2016} satisfy this
condition.
It is also satisfied if  the local fertility functions have a (uniformly) bounded compact
support (cf. \cite{hansen-reynaud-rivoirard2015} in the stationary case).
\begin{assumption}{LS}
\item   \label{ass:locally-stat-hawk-center-intensity-lip}
 Assume that, for some $\beta \in (0,1]$,
  \begin{align*}
\holderlsc:=\sup_{u\neq
  v}\frac{|\locstat{\lambda}_c(v)-\locstat{\lambda}_c(u)|}{|v-u|^\beta}<\infty\; .
\end{align*}
\item   \label{ass:locally-stat-hawk-unif-lip} Assume that, for some $\beta \in (0,1]$, $\ellnorm{\holderls}<\infty$, where
$$
\holderls (r) := \sup_{u\neq v}\frac{\left|\locstat{p}(r;v)-\locstat{p}(r;u)\right|}{|v-u|^\beta} \;.
$$
\end{assumption}
Assumptions \ref{ass:locally-stat-hawk-center-intensity-lip}
and~\ref{ass:locally-stat-hawk-unif-lip} can be interpreted as smoothness
conditions respectively on $\locstat{\lambda}_c$ and on
$\locstat{p}(\cdot;\cdot)$ with respect to its second argument.  Note also that
Assumptions~(\ref{eq:exp-decay-l1}) and~(\ref{eq:exp-decay-linfty}) imply in
particular Assumption (LS-4) of \cite{roueff-vonsachs-sansonnet2016} which we
recall here to be
\begin{align} \label{eq:assump-thm-log-laplace-zeta}
   \zetals[\infty]:=\zetaexpls[\infty]{0}=\sup_{u\in\Rset}\ellnorm{\locstat{p}(.;u)}[\infty]<\infty\;,\\
\label{eq:assump-thm-log-laplace-pSbeta}
\zetabet :=\sup_{u\in\Rset} \ellbetanorm{\locstat{p}(\cdot;u)}<\infty \;.
\end{align}
This can be seen simply by noting that $\zetals:=\zetaexpls{0}\leq\zetaexpls{d}$ and
$ \zetals[\infty]\leq \zetaexpls[\infty]{d}$ for all
$d\geq0$, with equality for $d=0$.  Similarly, $\zetaexpls{d}<\infty$ for some $d
>0$ implies $\zetabet<\infty$ for all $\beta>0$.

Hereafter all the given bounds are uniform upper bounds in the sense that they
hold uniformly over parameters $\locstat{\lambda}_c$ and $\locstat{p}$
satisfying the set of conditions
(\ref{eq:locally-stat-hawk-finite-intensity}),\ref{ass:locally-stat-hawk-center-intensity-lip},
\ref{ass:locally-stat-hawk-unif-lip},
and~(\ref{eq:assump-thm-log-laplace-zeta})
and~(\ref{eq:assump-thm-log-laplace-pSbeta}), as in
Theorems~\ref{thm:est_loc_mean_intensity}, or the more restrictive set of
conditions~\ref{ass:locally-stat-hawk-finite-expbounds},
\ref{ass:locally-stat-hawk-center-intensity-lip} and
\ref{ass:locally-stat-hawk-unif-lip}, as in Theorem~\ref{thm:total-bias}
and~\ref{thm:variance_loc_Bartlett_est}. More specifically, we use the
following conventions all along the paper. 
\begin{convention}[Symbol $\lesssim$]\label{conv:lesssim}
For two nonnegative sequences $a_T$ and
$b_T$ indexed by $T\geq1$, possibly depending on parameters
$\locstat{\lambda}_c$, $\locstat{p}$, $\btime$ and $\bfreq$, we use the
notation $a_T \lesssim b_T$ to denote that there exists a constant $C$ such
that, for all $\btime,\bfreq$ and $T$ satisfying certain conditions
$\mathcal{C}(\btime,\bfreq,T)$, we have $a_T \leq C\ b_T$ with $C$ only
depending on non-asymptotic quantities and constants such as $d$,
$\beta, \zetaexpls{d},\zetabet, \zetaexpls[\infty]{d}, \holderlsc,
\ellnorm{\holderls}, \ellnorm{\locstat{\lambda}_c}[\infty]$ and the two kernel
functions $\ktime$ and $\kfreq$.   
\end{convention}
The conditions
$\mathcal{C}(\btime,\bfreq,T)$ will be intersections of the following ones~:
  \begin{align}
    \label{eq:conditions-T-and-co-bias-density}
&T\geq 1\quad\text{and}\quad\btime\in(0,1]\;,\\    
    \label{eq:conditions-T-and-co-timebias-bartlett}
&\btime,\bfreq\in(0,1]\quad\text{and}\quad T\btime\bfreq\geq 1\;,    \\
        \label{eq:conditions-T-and-co-timebias-bartlett-expoterms}
&\btime\ln(T)\leq 1\;.
  \end{align}
  \begin{convention}[Constants $A_1,A_2$]\label{conv:As}
    We use $A_1,A_2$ to denote positive constants that
    can change from one expression to another but always satisfy $A_1^{-1}\lesssim1$ and
    $A_2\lesssim1$, using Convention~\ref{conv:lesssim}. In other words $A_1$
    and $A_2$ are positive constants which can be bounded from below and from
    above, respectively, using the constants appearing in the assumptions and
    the chosen kernels $\ktime$ and $\kfreq$.
  \end{convention}
  Convention~\ref{conv:As} will be useful to treat exponential terms in a
  simplified way, that is, without considering unnecessary constants; for
  instance, we can write $(\rme^{-A_1T})^2\leq\rme^{-A_1T}$ replacing $2A_1$ by
  $A_1$ in the second expression without affecting the property
  $A_1^{-1}\lesssim1$ required on $A_1$.
\subsection{Main results}
\label{sec:est_loc_mean_dens}
We can now state the main results of this contribution, whose proofs can be
found in Appendix~\ref{sec:proofs}.
For the bias and variance of the local mean density estimator
$\widehat{m}_{\btime}(u_0)$ we establish the following
result.
\begin{theorem}\label{thm:est_loc_mean_intensity}
Let the kernel $\ktime$ satisfy~\ref{ass:ktime-comp-support}. Assume conditions
  (\ref{eq:locally-stat-hawk-finite-intensity}),\ref{ass:locally-stat-hawk-center-intensity-lip},
  \ref{ass:locally-stat-hawk-unif-lip},
  and~(\ref{eq:assump-thm-log-laplace-zeta})
  and~(\ref{eq:assump-thm-log-laplace-pSbeta}) to hold. Then, for
  $\btime$ and $T$ satisfying~(\ref{eq:conditions-T-and-co-bias-density}), the
  bias of the local density estimator satisfies, for all $u_0\in\Rset$,
  \begin{equation}
  \label{eq:bias_loc_mean_dens_est}
  \left| \esp[ \widehat{m}_{\btime}(u_0)]- \locstat{m}_1(u_0)\right| \ \lesssim
  \ \btime^\beta + T^{-\beta} \;.
  \end{equation}
If moreover \ref{ass:locally-stat-hawk-finite-expbounds} holds, its variance satisfies
    \begin{equation}
  \label{eq:var_loc_mean_dens_est}
  \var\left(\widehat{m}_{\btime}(u_0)\right) \ \lesssim (T\btime)^{-1}\ .
  \end{equation}
\end{theorem}
Hence, $\widehat{m}_{\btime}(u_0)$ is shown to be a (mean-square) consistent
estimator of $\locstat{m}_1(u_0)$, and, optimizing the bias and variance
bounds, we get the ``usual''
mean-square error rate $T^{-\frac{2\beta}{2\beta+1}}$ for nonparametric curve estimation with an additive noise structure, achieved for a bandwidth $\btime
\sim T^{-\frac{1}{2\beta+1}}$.

We now treat the bias of the estimator
$\widehat{\gamma}_{\bfreq,\btime}(u_0;\omega_0)$ which can be decomposed as the
sum of 1) a bias in the time direction, namely, $\esp\
\widehat{\gamma}_{\bfreq,\btime}(u_0;\omega_0) -
\locstat{\gamma}_{\bfreq}(u_0;\omega_0)$ and 2) a bias in the frequency
direction, namely, $\locstat{\gamma}_{\bfreq}(u_0;\omega_0) -
\locstat{\gamma}(u_0;\omega_0)$.
\begin{theorem}
  \label{thm:total-bias}
  Let the kernels $\ktime$ and $\kfreq$ satisfy~\ref{ass:ktime-comp-support}
  and~\ref{ass:kfreq-comp-support}.  Assume
  conditions~\ref{ass:locally-stat-hawk-finite-expbounds},
  \ref{ass:locally-stat-hawk-center-intensity-lip}, and
  \ref{ass:locally-stat-hawk-unif-lip} to hold.
  Then, for all $\btime$, $\bfreq$ and $T$ satisfying   (\ref{eq:conditions-T-and-co-timebias-bartlett})
  and~(\ref{eq:conditions-T-and-co-timebias-bartlett-expoterms})
  and for all $u_0,\ \omega_0 \in \Rset$, we have
  \begin{equation}\label{eq:total_bias}
\left|\esp[ \widehat{\gamma}_{\bfreq,\btime}(u_0;\omega_0) ]-
  \locstat{\gamma}_{\bfreq}(u_0;\omega_0) \right|
 \lesssim\  \ \btime^\beta + \btime^{2\beta} \bfreq^{-1} \ +\ (T\btime\bfreq)^{-1} \ .
\end{equation}
If moreover the squared modulus $|\fkfreq(\omega)|^2$ of the Fourier transform
of the kernel $\kfreq$ satisfies
\begin{equation}
  \label{eq:addtional-cond-Kkernek} \int
\omega^2|\fkfreq(\omega)|^2\rmd\omega<\infty\quad\text{and}\quad\int
\omega|\fkfreq(\omega)|^2\rmd\omega=0\;,
\end{equation}
the ``bias in frequency direction'' fulfills for $\bfreq\in(0,1]$,
  \begin{equation}\label{eq:freq_bias}
 \left| \locstat{\gamma}_{\bfreq}(u_0;\omega_0) -   \locstat{\gamma}(u_0;\omega_0) \right| \lesssim \bfreq^2\ .
 \end{equation}
\end{theorem}
\begin{remark}
  Condition~(\ref{eq:addtional-cond-Kkernek}) is automatically satisfied if
  $\kfreq$ is compactly supported, real valued and even, and admits an
  $L^2$ derivative, such as the triangle shape kernel. 
\end{remark}
We already observe here that for the estimator
$\widehat{\gamma}_{\bfreq,\btime}(u_0;\omega_0)$ to be asymptotically unbiased,
equations~(\ref{eq:total_bias}) and ~(\ref{eq:freq_bias}) require the following
conditions on the choice of the two bandwidths $\btime$ and $\bfreq$ to be
fulfilled:
\begin{equation*}
T\btime\bfreq \to \infty\;,\quad \bfreq\to0\, \quad {\mbox{and}}\;\; \bfreq^{-1}\btime^{2\beta} \to 0\ .
\end{equation*}
Note in
particular that these conditions for an  asymptotically unbiased estimator
imply those required for the feasibility of the estimator in
Remark~\ref{rem:finite-set-of observations} ($\btime\to0$ and $T\bfreq\to\infty$). 
 
We shall discuss
possible compatible bandwidth choices below, following
the treatment of the variance of this estimator. 
\begin{theorem}
  \label{thm:variance_loc_Bartlett_est}
  Let the kernels $\ktime$ and $\kfreq$ satisfy~\ref{ass:ktime-comp-support}
  and~\ref{ass:kfreq-comp-support}.  Assume
  conditions~\ref{ass:locally-stat-hawk-finite-expbounds},
  \ref{ass:locally-stat-hawk-center-intensity-lip} and
  \ref{ass:locally-stat-hawk-unif-lip} to hold. Then, for all $\btime,\bfreq,T$
  satisfying~(\ref{eq:conditions-T-and-co-timebias-bartlett}), and for all
  $u_0,\ \omega_0 \in \Rset$, we have 
  \begin{equation}\label{eq:variance_loc_Bartlett_est}
\var\left(\widehat{\gamma}_{\bfreq,\btime}(u_0;\omega_0)\right) \ \lesssim\  \
(T\btime\bfreq)^{-1} +  \btime^{2\beta}\,\left(\btime^{\beta} \bfreq^{-1}\right)^2\ .
\end{equation}
\end{theorem}

\subsection{Immediate consequences and related works}

In order to optimize bandwidth choices in time and in frequency to derive an
optimal rate of MSE-consistency of the estimator
$\widehat{\gamma}_{\bfreq,\btime}(u_0;\omega_0)$ for a given time-frequency
point $(u_0;\omega_0)$ we observe, by
equations~(\ref{eq:total_bias}),~(\ref{eq:freq_bias})
and~(\ref{eq:variance_loc_Bartlett_est}), that the MSE satisfies
\[
\mbox{MSE}\left(\widehat{\gamma}_{\bfreq,\btime}(u_0;\omega_0)\right)\ \lesssim\ \bfreq^4\ + \btime^{2\beta}\ +  (T\btime\bfreq)^{-1}\ +  \left(\btime^{2\beta} \bfreq^{-1}\right)^2\ .
\]
The MSE-rate $T^{-\frac{4\beta}{5\beta+2}}\ $ is achieved by optimizing this
upper bound, that is, by imposing
$\bfreq^4 \sim \btime^{2\beta} \sim (T\btime\bfreq)^{-1}\ $ (leading to
$\left(\btime^{2\beta} \bfreq^{-1}\right)^2 \sim \bfreq^6\ $), and hence by
setting $\ \btime \sim T^{-\frac{2}{2+5\beta}}\ $ and
$\ \bfreq \sim T^{-\frac{\beta}{2+5\beta}}$. The fact
that this rate bound is obtained by balancing the two squared bias terms
$\btime^{2\beta}$ and $\bfreq^4$ with the variance term $(T\btime\bfreq)^{-1}$
indicates that all the other terms appearing in the upper
bounds~(\ref{eq:total_bias}) and~(\ref{eq:variance_loc_Bartlett_est}) are
negligible. Thus, since the bias terms $\btime^{\beta}$ and $\bfreq^2$ and the
variance term $(T\btime\bfreq)^{-1}$ correspond to the usual bias and variance
rates of a kernel estimator of a local spectral density estimator (see
\cite[Example~4.2]{dahlhaus-2009} for locally stationary linear time series),
it is clear that the obtained rate is sharp for this moment estimator under our
assumptions.  Note also that the MSE-rate $T^{-\frac{4\beta}{5\beta+2}}\ $
corresponds to the minimax lower bound for evolutionary spectrum estimation
established in \cite[Theorem~2.1]{neumann-vonsachs1997} in the Sobolev space
$W_{\infty,\infty}^{\beta,2}$. Although insightful, the comparison is not
completely rigorous as their model for establishing this lower bound is a
benchmark for a class of non-stationary time series and the MSE they consider
is integrated. Therefore, a particularly interesting problem for future work
would be to derive (hopefully large) classes of locally stationary point processes on which our
estimator achieves the minimax rate.

In the same line of thoughts about the performance of our kernel estimator, the
question naturally arises about the data-driven choice of the bandwidths
$\btime$ and $\bfreq$. This question of bandwidth selection in the context of
locally stationary time series has been addressed only recently, see
\cite{giraud-roueff-sanchez-aos2015,richter-dahlhaus2017} for adaptive
prediction and parameter curve estimation, respectively. The problem of
adaptive kernel estimation of the local spectral density for locally stationary
time series has been more specifically addressed in
\cite{vandelft-eichler-2017}, where a practical approach is derived and
studied. It is mainly based on the central limit theorem established in
\cite[Example~4.2]{dahlhaus-2009} for this estimator.  This methodology could
be adapted to the case of locally stationary point processes. A first step in
this direction would be to establish a central limit theorem for our estimator
$\widehat{\gamma}_{\bfreq,\btime}(u_0,\omega_0)$ at a given time-frequency
point, which is also left for future work. Note also that, in practical time
frequency analysis, the bandwidths $\btime$ and $\bfreq$ are often chosen
having in mind a physical interpretation. For instance, in our real data
example of Section~\ref{sec:numer-exper}, on each day, transaction data
is collected between 9:00 AM and 5:30 PM, hence over $8.5$ hours. Our choice 
$\btime=.15$ and $\bfreq=.005$ corresponds to saying that we consider finance
transactions data as roughly stationary over $8.5\times.15\approx1$ hour and
16 minutes of time, and that the spectrum obtained from such data
has maximal frequency resolution  $.005$ Hz (we may distinguish between two
periodic behaviors present in the data only if their frequencies differ by at
least this value).

To conclude this section, let us discuss whether our approach could be used for
the parameter estimation of locally stationary Hawkes processes. In a fully
non-parametric approach, one would be interested in estimating the two
(possibly smooth or sparse) unknown functions that are the baseline intensity
function $u\mapsto\locstat{\lambda}_c(u)$ and the local fertility function
$(s,u)\mapsto\locstat{p}(s;u)$. In a parametric approach, one would assume
these functions to depend on an unknown finite dimensional parameter $\theta$
in a (known) form $u\mapsto\locstat{\lambda}_c(u|\theta)$ and
$(s,u)\mapsto\locstat{p}(s;u|\theta)$, and try to estimate $\theta$. An
intermediate approach, proposed in
\cite[Section~5.1]{roueff-vonsachs-sansonnet2016} to derive simple examples of
locally stationary Hawkes processes, is to consider a parametric stationary
model for the fertility function, say $s\mapsto\stat{p}(s|\theta)$ for
$\theta\in\Theta\subset\rset^d$, and to deduce a local one of the form
$(s,u)\mapsto\locstat{p}(s;u)=\stat{p}(s|\boldsymbol{\theta}(u))$, where now
$\boldsymbol{\theta}$ is a $\Theta$-valued function of the absolute time.  In
all these cases, the estimators $\widehat{m}_{\btime}$ and
$\widehat{\gamma}_{\bfreq,\btime}$ could be used as empirical moments to
estimate $\locstat{\lambda}_c$, $\locstat{p}$, $\theta$, or the curve
$\boldsymbol{\theta}$.  Since our estimators are consistent, this method would
in principle work whenever the unknown quantities to estimate can be deduced
from the local mean density $\locstat{m}_1$ and local Bartlett spectrum
$\locstat{\gamma}$ through Relations~(\ref{eq:local-mean})
and~(\ref{eq:local-bartlett-spectral-density}), respectively.  More direct
methods to estimate the parameters of interest for non-stationary Hawkes
processes have been proposed in
\cite{chen-hall-2013,chen-hall-2016,mammen2017}.  In the first two references,
only the baseline intensity is time varying, and a different asymptotic setting
is considered, where this baseline intensity tends to infinity through a
multiplicative constant. In the fully non-parametric case, an identifiability
problem is pointed out in \cite[Section~2.2]{chen-hall-2016}. We do not have
this problem in our asymptotic scheme, since~(\ref{eq:local-mean})
and~(\ref{eq:local-bartlett-spectral-density}) show that the base line
intensity $\locstat{\lambda}_c(u)$ can be completely identified from the local
mean density $\locstat{m}_1(u)$ and the local Bartlett spectrum
$\locstat{\gamma}(u;\cdot)$ alone using the formula
$$
\locstat{\lambda}_c(u)=\locstat{m}_1(u)\left(\frac{\locstat{m}_1(u)}{2\pi \locstat{\gamma}(u;0)}\right)^{1/2}\;.
$$
The model considered in \cite{mammen2017} is a multivariate version of the
locally stationary Hawkes process with the same asymptotic setting as ours, and
additional assumptions on the (multivariate) fertility function.  Both the
baseline intensity and the (time varying) fertility function are estimated in a
non-parametric fashion using a direct method based on localized mean square
regression and a decomposition of the local fertility function on a B-spline
base. Such methods should be more efficient than using the local mean density
and Bartlett spectrum to build moment estimators of these parameters, since
they rely on the intrinsic auto-regression structure of the underlying
process. In contrast, as far as the time-frequency analysis is concerned, which
is the main focus of our contribution, the estimator that we propose should be
relevant to estimate the local Bartlett spectrum beyond the case of locally
stationary Hawkes processes, namely, for any locally stationary point process
for which the general formula~(\ref{eq:regularized})
and~(\ref{eq:regularized-new}) make sense.

\subsection{Main ideas of the proofs}
\subsubsection{Local approximations of moments}
\label{sec:appr-results-locally}

An essential step for treating the bias terms is to be able to approximate, as
$T\to\infty$, in the neighborhood of $uT$, the first and second moments of $N_T$ by
that of the local stationary approximation $N(\cdot;u)$ defined as in
Section~\ref{sec:local-mean-density-bartlett}. We first state the two
approximations that directly follow
\cite[Theorem~4]{roueff-vonsachs-sansonnet2016} with $m=1,2$.

\begin{theorem}\label{thm:mean-var-approx}
  Let $\beta\in(0,1]$ and let $(N_T)_{T\geq1}$ be a locally stationary Hawkes
  process satisfying conditions
  (\ref{eq:locally-stat-hawk-finite-intensity}),\ref{ass:locally-stat-hawk-center-intensity-lip},
  \ref{ass:locally-stat-hawk-unif-lip}, ~(\ref{eq:assump-thm-log-laplace-zeta})
  and~(\ref{eq:assump-thm-log-laplace-pSbeta}). Let $(N(\cdot;u))_{u\in\rset}$
  be the collection of stationary Hawkes process defined as in
  Section~\ref{sec:local-mean-density-bartlett}. 
  Then, for all $T\geq1$, $u\in\rset$ and all bounded integrable functions $g$, we
  have
  \begin{align}
    \label{eq:mean-approx-hawkes-local}
    \left|\esp[N_T(S^{-Tu}g)] - \esp[N(g;u)]\right| &\lesssim
    \left(\ellnorm{g}+\ellbetanorm{g}\right)\;
T^{-\beta}\ ,\\
   \label{eq:var-approx-hawkes-local}
   \left|\var\left(N_T(S^{-Tu}g))\right)-\var\left(N(g;u)\right)\right|&
\lesssim
\left(\ellnorm{g}+\ellnorm{g}[\infty]\right)
\left(\ellnorm{g}+\ellbetanorm{g}\right)
\;T^{-\beta}\;.
\end{align}
\end{theorem}
The control of the bias in Theorem~\ref{thm:est_loc_mean_intensity} directly
follows from~(\ref{eq:mean-approx-hawkes-local}). However it turns out
that~(\ref{eq:var-approx-hawkes-local}) is not sharp enough to control the bias
of the local Bartlett spectrum and thus to obtain the expected convergence
rate. The basic reason is that it involves $L^1$ (weighted) norms of $g$ in the
upper bound instead of $L^2$ norms.  In order to recover the correct rates of
convergence, we rely on the following new result where the $L^1$ (weighted)
norms are indeed replaced by $L^2$ (weighted) norms, or, to be more precise,
where the remaining weighted $L^1$-norms are compensated by an exponentially
decreasing term in $T$, and will thus turn out to be negligible.
\begin{theorem}\label{thm:new-var-approx}
  Let $\beta\in(0,1]$ and let $(N_T)_{T\geq1}$ be a locally stationary Hawkes
  process satisfying conditions
  \ref{ass:locally-stat-hawk-finite-expbounds},\ref{ass:locally-stat-hawk-center-intensity-lip}
  and \ref{ass:locally-stat-hawk-unif-lip}. Let $(N(\cdot;u))_{u\in\rset}$ be
  the collection of stationary Hawkes process defined as in
  Section~\ref{sec:local-mean-density-bartlett}.  Then, for all bounded and compactly supported 
  functions $g$, for all $T\geq1$ and $u\in\rset$, 
\begin{align}
  \label{eq:bound-of-var-ell2}
  &\left| \var(N_T(g)) \right|\lesssim \ellnorm{g}[2]^2+\rme^{- A_1 T}\ellexpnorm{g}[1][d]^2\;,\\
\nonumber
\big|\var\left(N_T(S^{-Tu}g))\right)&-\var\left(N(g;u)\right)\big|\\
   \label{eq:new-var-approx-hawkes-local}
&\lesssim
\left\{\ellnorm{g}[2]^2+\rme^{- A_1 T}\ellexpnorm{g}[1][d]^2+\ellbetanorm{g}[2][\beta]\left(\ellnorm{g}[2]+\rme^{- A_1
  T}\ellexpnorm{g}[1][d]\right)\right\}\, T^{-\beta}
  \; .
\end{align}
\end{theorem}
The proof of this theorem can be found in Appendix~\ref{sec:proof_Thm4}. 
To obtain this new result, we crucially rely on the
assumption~\ref{ass:locally-stat-hawk-finite-expbounds} where controls in
exponentially weighted norms (based on
$\sup_u\ellexpnorm{\locstat{p}(\cdot;u))}[q][d]$ for $q=1,\infty$ and some
$d>0$) are assumed to strengthen the assumptions
(\ref{eq:locally-stat-hawk-finite-intensity}),
(\ref{eq:assump-thm-log-laplace-zeta})
and~(\ref{eq:assump-thm-log-laplace-pSbeta}).

\subsubsection{Bias and approximate centering}
\label{sec:bias-appr-cent}

To control all the error terms, we found useful to introduce a centered
version of $N_T$ with a centering term corresponding to its asymptotic
deterministic version. Recalling that $N_T$ behaves in a neighborhood
of $T\,u_0$ as $N(\cdot; u_0)$ and that this process admits the mean intensity
denoted by $\locstat{m}_1(u_0)$, we define, for any test
function $f$,
\begin{align}\label{eq:approx-centering}
\overline{N}_T(f;u_0):= N_T(f) - \esp[N(f;u_0)]
=\int f(s) \; [N_T(\rmd s)-\locstat{m}_1(u_0)\,\rmd s]
\;.
\end{align}
It is important to note that this ``approximate''  centering depends on an absolute location
$u_0$ as it is a good approximation of $\esp[ N_T(f)]$ only for $f$ localized
in a neighborhood of $T\,u_0$. 
Let us apply this definition. By~(\ref{eq:intensity_est}), since $\ktbtime$ integrates to 1,
the error of the local mean density estimator can directly be expressed as
\begin{align}\label{eq:bias-mean-Nbar-expr}
\widehat{m}_{\btime}(u_0)- \locstat{m}_1(u_0) =\overline{N}_T(S^{-T
  u_0}\ktbtime;u_0) \;. 
\end{align}
Hence controlling the bias of this estimator directly amounts to evaluating the
quality of the above centering.

The treatment of the bias of the local Bartlett spectrum is a bit more involved
since, as often for spectral estimators, the empirical centering
term requires a specific attention. This term appears inside the negated square modulus
of the right-hand side of~(\ref{eq:kernel-est}). To see why it is indeed a
centering term, observe that, using  $\int\ktbtime=1$, we can write
$$
\widehat{\gamma}_{\bfreq,\btime}(u_0;\omega_0) =
\widehat{E}\left[\rho(N_T(S^{-Tu_0}\kbofreq));\kbtime\right]
$$
with now $\rho$ defined, for any test function $f$, as the ``centered'' squared modulus
\begin{align*}
\rho(N_T(f))=\left|N_T(f)-\widehat{E}\left[N_T(f);\kbtime\right]\right|^2\;.  
\end{align*}
Using that $\int\kbtime=1$, the centering in~(\ref{eq:approx-centering}) can be
introduced within this definition of $\rho$, leading to
$$
\rho(N_T(f)) =
\left|\overline{N}_T(f;u_0)
-\widehat{E}\left[\overline{N}_T(f;u_0);\kbtime\right]\right|^2\; .
$$
By comparison with the previous expression of
$\widehat{\gamma}_{\bfreq,\btime}(u_0;\omega_0)$,
we easily get an expression of the local Bartlett spectrum estimator based on
this centered version of $N_T$, namely,
\begin{align}\label{eq:est_two_terms_decomp}
\widehat{\gamma}_{\bfreq,\btime}(u_0;\omega_0)
 =\widehat{E}\left[|\overline{N}_T(f;u_0)|^2;\kbtime\right]-
\left|\widehat{E}\left[\overline{N}_T(f;u_0);\kbtime\right] \right|^2 \ ,
\end{align}
where analogously to~(\ref{eq:general_moment_est}), we denote, for the test
function $f=S^{-Tu_0}\kbofreq$,
\begin{align}
\label{eq:est-first-order-square-centered}
\widehat{E}\left(\left|\overline{N}_T(f;u_0)\right|^2;\kbtime\right)& := \int
\left|\overline{N}_T(f(\cdot-t);u_0)\right|^2\ 
\ktbtime(t)\ \rmd t\ ,\\
\nonumber
\widehat{E}\left(\overline{N}_T(f;u_0);\kbtime\right)\ &:= \int
\overline{N}_T(f(\cdot-t);u_0)\ 
\ktbtime(t)\ \rmd t\;.
\end{align}
In fact, using $\int\ktime=1$, $f$ integrable and interchanging the order of integration, we
immediately get the simplification
\begin{equation}
\label{eq:est-first-order-centered}
  \widehat{E}\left(\overline{N}_T(f;u_0);\kbtime\right)= \overline{N}_T(f\ast\ktbtime;u_0)\;,
\end{equation}
where we used the convolution product $\ast$. The advantage of the new
expression~(\ref{eq:est_two_terms_decomp}) in contrast to the
original~(\ref{eq:kernel-est}) is that now we expect the negated square modulus
to be of negligible order. To see why, consider for instance the bias
in~(\ref{eq:total_bias}), for the control of which we need to bound, as the
second term of~(\ref{eq:est_two_terms_decomp}),
\begin{equation}
\label{eq:centering-term-bartlesst-expect-decomp}
\esp\left[\left|\widehat{E}\left(\overline{N}_T(f;u_0);\kbtime\right)
  \right|^2\right]
=\var\left(N_T(f\ast\ktbtime)\right)
+ \left|\esp\left[\overline{N}_T(f\ast\ktbtime;u_0)\right]\right|^2\;.
\end{equation}
(Here and in the following we repeatedly use the fact that $N_T$ and
$\overline{N}_T(\cdot;u_0)$ have the same variance).
Finally, the control of the bias term in~(\ref{eq:total_bias}) now requires to evaluate 
$$
\esp\left[\widehat{E}\left[|\overline{N}_T(S^{-Tu_0}\kbofreq;u_0)|^2;\kbtime\right]\right]
- \locstat{\gamma}_{\bfreq}(u_0;\omega_0)\;,
$$
which is again decomposed as a main term
\begin{equation}
\label{eq:bias-bartlett-main-term}
\int
\left(\var\left(N_T(S^{-Tu_0}\kbofreq(\cdot-t))\right)-\locstat{\gamma}_{\bfreq}(u_0;\omega_0)\right)\ktbtime(t)\
\rmd t \;,
\end{equation}
added to a negligible (because involving a squared bias) term 
\begin{equation}
\label{eq:bias-bartlett-neg-term2}
\int \left|\esp[\overline{N}_T(S^{-Tu_0}\kbofreq(\cdot-t);u_0)]\right|^2\ktbtime(t)\ \rmd t \;.
\end{equation}

\subsubsection{Variance terms and exact centering}
\label{sec:variance-terms-exact}

The variance of the local mean density estimator directly requires to control
the variance of $N_T(f)$ for given test functions $f$. This requires new deviation bounds for
non-stationary Hawkes processes. By deviation bounds we here mean that we bound
the moments of
\begin{equation}
  \label{eq:true-centering-def}
  \overline{\overline{N}}_T(h)\ := N_T(h) - \esp [N_T(h)]\;,
\end{equation}
where $h$ is an appropriate test function. New results in this direction are
gathered in Section~\ref{sec:dev-bounds}, where the dependence structure of
non-stationary Hawkes processes is investigated, leading to
the appropriate control of such
moments in Proposition~\ref{prop:burkh-ineq}.
For instance the moment of order 2 directly provides the adequate bound for the
variance of the local mean density estimator
\begin{equation}
\label{eq:var-mea-est-expr-barbar}
  \var\left(\widehat{m}_{\btime}(u_0)\right)=
\var\left(N_T(S^{-Tu_0}\ktbtime)\right)=\left\| \overline{\overline{N}}_T(h) \right\|_2^2 \;,
\end{equation}
where we use the notation introduced in~(\ref{eq:notation-q-norm-Lq}).

Now we turn our attention to the estimator of the local Bartlett spectrum. The
control of the moments of $\overline{\overline{N}}_T$ will essentially be used
to approximate $\widehat{\gamma}_{\bfreq,\btime}(u_0;\omega_0)$ by
\begin{equation}
  \label{eq:exact-centering-approx}
  \widetilde{\gamma}_{\bfreq,\btime}(u_0;\omega_0):=
\int \left|\overline{\overline{N}}_T(S^{-Tu_0}\kbofreq(\cdot-t))\right|^2 \ktbtime(t) \;\rmd t\;.
\end{equation}
In contrast to the centering used in $\overline{N}_T(\cdot;u)$ for controlling
the bias (a centering with respect to $\esp [N(h;u)]$) for some absolute location
$u$, here the term $\esp [N_T(h)]$ is no longer invariant as $h$ is
shifted. This is why this centering cannot be used as a direct decomposition of
estimator $\widehat{\gamma}_{\bfreq,\btime}(u_0;\omega_0)$ as
in~(\ref{eq:est_two_terms_decomp}). Instead we use a bound on
the error of approximating $\widehat{\gamma}_{\bfreq,\btime}(u_0;\omega_0)$ by
$\widetilde{\gamma}_{\bfreq,\btime}(u_0;\omega_0)$, see
Lemma~\ref{lem:approximation-exact-centering}. 

Finally the variance of the local Bartlett estimator is obtained by controlling
the variance of $\widetilde{\gamma}_{\bfreq,\btime}(u_0;\omega_0)$
(Lemma~\ref{lem:var_lead_term}), which in turn relies on a bound of 
$$
\cov\left(\left|\overline{\overline{N}}_T(h_1)\right|^2,\left|\overline{\overline{N}}_T(h_2)\right|^2\right)
$$ 
for test functions $h_1$ and $h_2$, which is derived in
Corollary~\ref{cor:cov-squares}.

\section{Numerical experiments}
\label{sec:numer-exper}

The numerical experiments in \cite{roueff-vonsachs-sansonnet2016} show that the
estimators $\widehat{m}_{\btime}$ and $\widehat{\gamma}_{\bfreq,\btime}$ are
able to reproduce the theoretical local mean density and local Bartlett spectrum
on simulated locally stationary Hawkes processes.  Here we consider a real data
set containing the transaction times of the two assets ESSI.PA (Essilor International SA) and TOTF.PA  (Total SA) over
61 days scattered in February, June and November 2013.
 We computed the local
mean density and Bartlett spectrum estimators, say $\widehat{m}_{\btime}^{(i)}$ and
$\widehat{\gamma}_{\bfreq,\btime}^{(i)}$ for each day $i\in\{1,\dots,61\}$ over the regular
opening hours of the Paris stock exchange market, that is between 9:00 a.m. and
5:30 p.m., Paris local time. The estimators are computed with the following
kernels~: $\ktime$ is the triangle kernel and $\kfreq$ is the Epanechnikov kernel, both
with supports $[-.5,.5]$. The chosen bandwidth parameters are given by
$$
\btime=.15\;,\quad \bfreq=.005 \,\mathrm{Hz}\;.
$$ 
The above unit for $\btime$ is absolute time, that is, 1 unit corresponds to the overall
duration of observation $T=8.5$ hours, hence in real time, 
$\btime=.15*8.5$ hours, which makes 1 hour, 16 minutes and 30 seconds. 
We thus obtain for each asset 61 local mean density and local Bartlett spectrum
estimates. Our goal here is to illustrate the time frequency analysis of such
point processes data sets. The obtained results are quite different from one
day to another, which can be expected on such real data. However, in the
following, we propose to comment on the local mean densities and Bartlett spectra
obtained for the two assets by averaging over the available 61 days,
$$
\widehat{m}_{\btime}^{(\tiny{Av})}=\frac1{61}\sum_{i=1}^{61}\widehat{m}_{\btime}^{(i)}
\quad\text{and}\quad
\widehat{\gamma}_{\bfreq,\btime}^{(\tiny{Av})}=\frac1{61}\sum_{i=1}^{61}\widehat{\gamma}_{\bfreq,\btime}^{(i)} \;.
$$  
Moreover we computed a \emph{Poisson-normalized} local Bartlett spectrum of
these averaged estimates defined by
$$
\widehat{\gamma}_{\bfreq,\btime}^{(\tiny{Pn})}(\omega;u)=\frac{2\pi\,\widehat{\gamma}_{\bfreq,\btime}^{(\tiny{Av})}(\omega;u)}{\widehat{m}_{\btime}^{(\tiny{Av})}(u)}\,\qquad u\in\rset,\omega\in\rset\;.
$$
Note that, in the case of a nonhomogeneous Poisson process, the local mean density
and Bartlett density satisfy 
$$
\locstat{\gamma}(\omega;u)=\frac{\locstat{m}_1(u)}{2\pi}\;.
$$
This is indeed given by~(\ref{eq:local-bartlett-spectral-density}) with a local
fertility function to be identically zero, $\locstat{p}(\cdot;u)\equiv0$. 

In Figures~\ref{fig:essi-av} and~\ref{fig:totf-av} we display the resulting
estimators for the two assets ESSI.PA and TOTF.PA, respectively. Note that the scaling of the
$y$-axis of the averaged local mean densities (top plots) is not the
same. The transaction rate of ESSI.PA evolves around 0.1 transactions per
second while that of the more liquid TOTF.PA around twice as much.  The local
Bartlett spectrum estimator $\widehat{\gamma}_{\bfreq,\btime}(\omega;u)$ is
computed over frequencies $\omega$ ranging between 0 and .1 Hz. This means that
only cyclic behaviors with periods larger than 10 seconds are visible. As for
the local mean density plots, note that the color scales of the averaged local
Bartlett spectra are different for the two assets. 

It is interesting to observe that, despite these differences of orders of
magnitude, the shapes of the averaged local mean densities and that of the
averaged local Bartlett spectra bear some similarities. Namely the averaged
mean density is decreasing in the morning, although a sharp increase occurs around
11:00 a.m. and a drop during the lunch break. It then increases steadily
during the afternoon with a sharper increase around 3:30 p.m., which
corresponds to the opening time of the New York stock exchange market. The
maximal averaged mean density is reached at the closing time. As for the averaged
Bartlett spectrum, it is interesting to note that the shape of the spectrum
along the frequencies varies significantly along the day. During the increases
of mean density preceding and following lunch break, the spectrum concentrates at
low frequencies, while the spectrum, although still favoring low frequencies,
is more balanced during the increase following the opening of the NYSE
market. Finally, it is interesting to observe that the Poisson-normalized
Bartlett spectra take the highest values during the two one hour long periods
surrounding the lunch break. It indicates that, in contrast to these two
periods, the increase of the (nonnormalized) Bartlett spectrum toward the end
of the day can be interpreted merely as a consequence of the increase of the
local mean density rather than a departure from the Poisson behavior.  Also
observe that the Poisson-normalized Bartlett spectra are always larger than 1.
Assuming a locally stationary Hawkes process for this data, this could be
interpreted, according to Formula~(\ref{eq:local-bartlett-spectral-density}), 
as 
$$
\left|1-\locstat{\fourierp}(\omega;u)\ \right| < 1 \;,
$$
where $\locstat{\fourierp}(\cdot;u)$ is the Fourier transform of the local
fertility function $\locstat{p}(\cdot;u)$.

A sensible conclusion of this analysis is that it advocates for more
involved models than a simple non-homogeneous Poisson process for transaction
data. In particular, locally stationary Hawkes processes as assumed in this
work are better adapted to such data sets, not only because the local Bartlett
spectrum is not constant along the frequencies but also because its shape
varies along the time, a feature that could not be obtained by using time varying
baseline intensity with a fertility function constant over the time, as used in
\cite{chen-hall-2013}. This conclusion is of particular interest in relation
with Examples~2.3 (iii) and (iv) described in \cite{dahlhaus16} for modeling
the volatility of high frequency financial data.

\begin{figure}[h]
  \centering
\includegraphics[width=.9\textwidth]{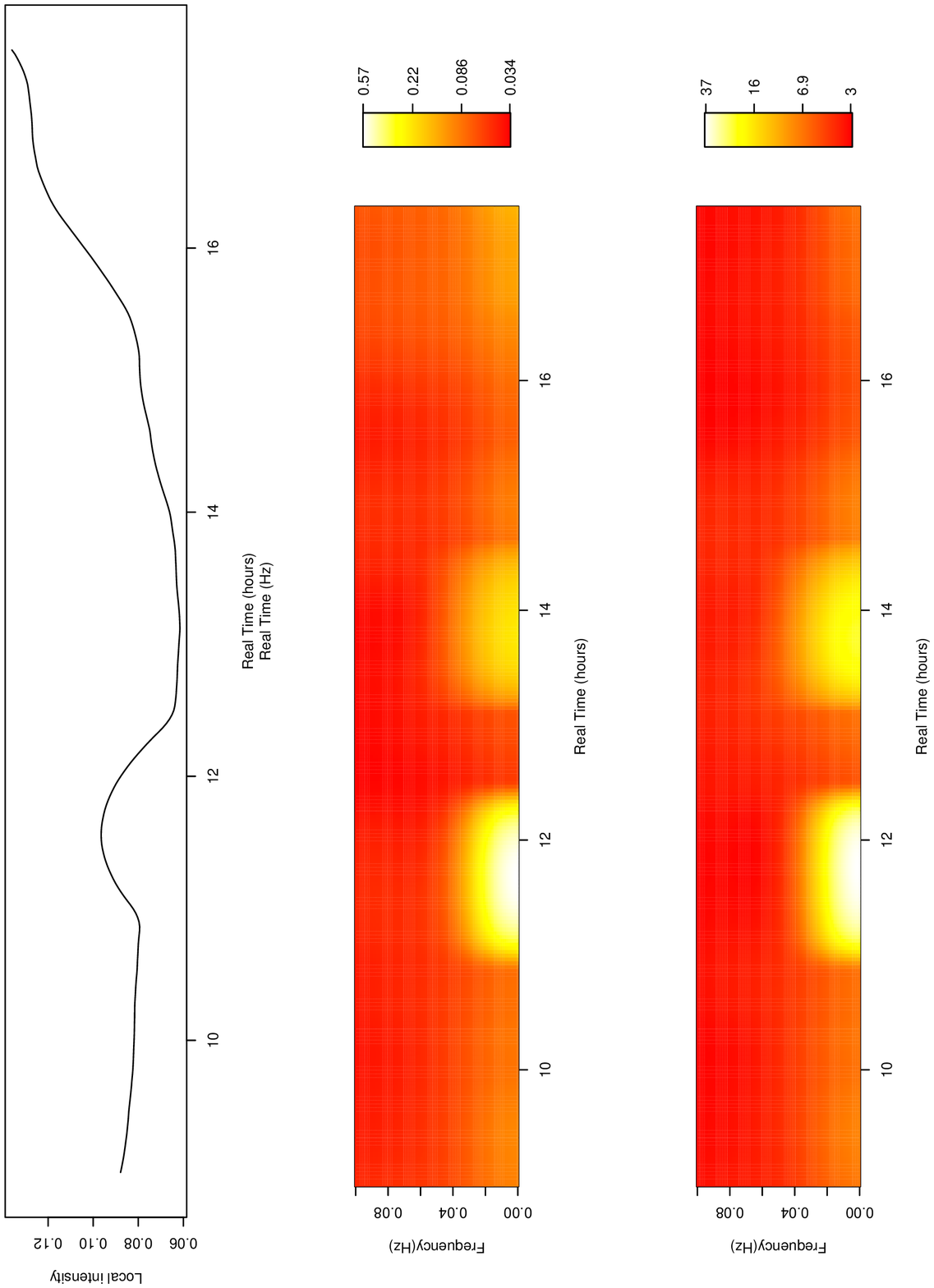} \\ 
\caption{Averaged local mean density (top) and Bartlett spectrum, nonnormalized (middle) and
  Poisson-normalized (bottom), for
  ESSI.PA transaction data.}
  \label{fig:essi-av}
\end{figure}

\begin{figure}[h]
  \centering
\includegraphics[width=.9\textwidth]{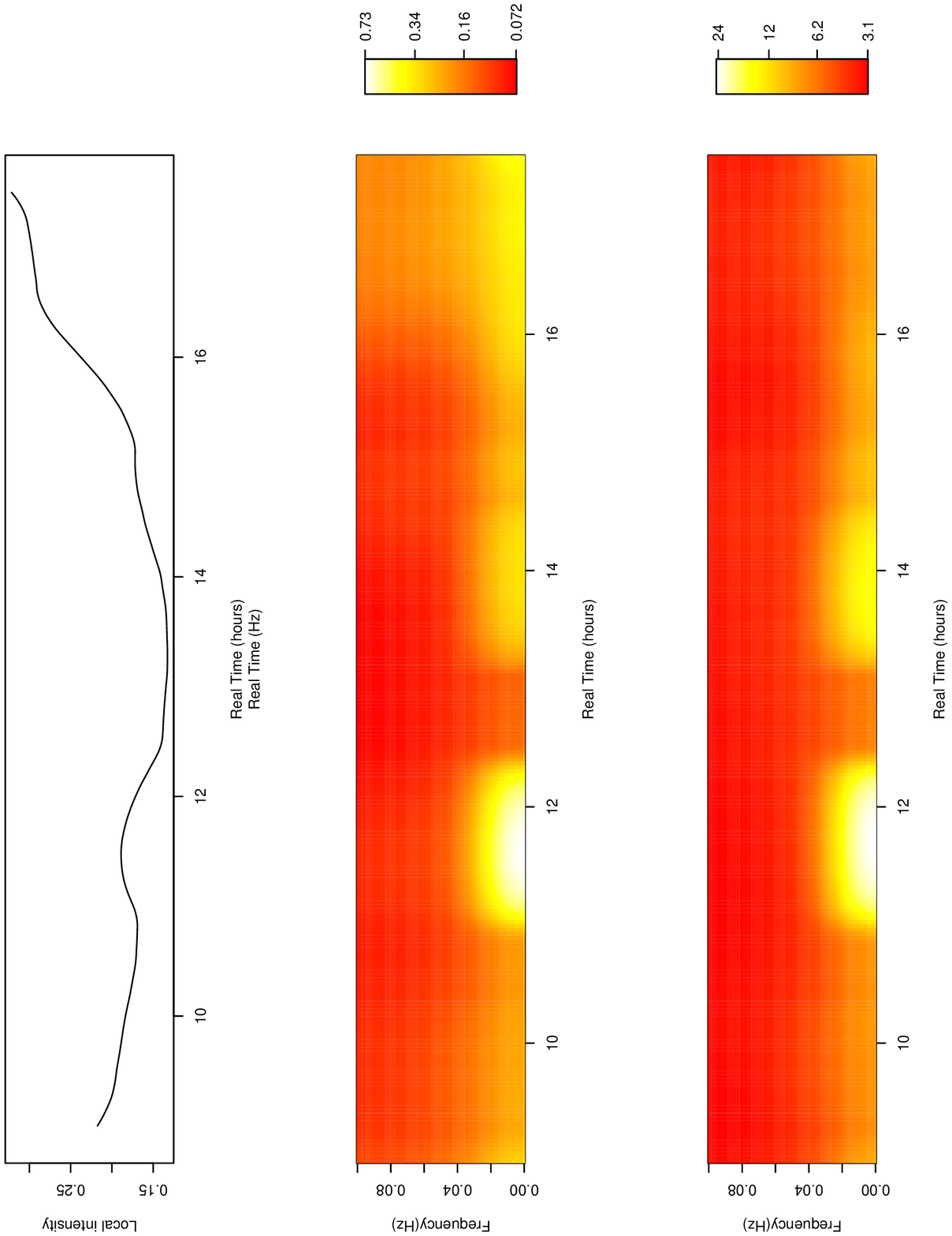} \\ 
\caption{Averaged local mean density (top), Bartlett spectrum, nonnormalized (middle)   and
  Poisson-normalized (bottom), for
  TOTF.PA transaction data.}
  \label{fig:totf-av}
\end{figure}

\section{Deviation bounds for non-stationary Hawkes processes}
\label{sec:dev-bounds}

We now derive new results required for treating the variance of the
studied 
estimators. In contrast to Poisson processes, we can not rely on the
independence of the process on disjoint sets.  To the best of our knowledge, the most
advanced results on deviation bounds of Hawkes processes are to be found in
\cite{reynaud-b-roy-2006} and only apply to stationary Hawkes processes with
compactly supported fertility functions.  Here we consider non-stationary
Hawkes processes with exponentially decreasing local fertility functions. The
generalization to non-stationary processes requires a specific approach
rather than a mere adaption of \cite{reynaud-b-roy-2006}.

Although we are here motivated by the study of the variance of the local mean density and
Bartlett estimators, we believe that the results contained in this section are
of broader interest, as they can serve more generally to understand the
dependence structure of non-stationary Hawkes processes.

\subsection{Setting}
\label{sec:setting}

Recall that~(\ref{eq:cond-density-nonstat}) is our minimal condition for
defining the non-stationary Hawkes process $N$ with immigrant intensity
function $\lambda_c(\cdot)$ and varying fertility function $p(\cdot;\cdot)$.
The cluster construction of a Hawkes process relies on conditioning on the
realization of a so-called \emph{center process}, $N_c$ a Poisson point process
(PPP) with intensity $\lambda_c$ on $\rset$, which describes spontaneous appearing of the immigrants. At
each \emph{center point} $t^c$ of $N_c$, a point process $N(\cdot|t^c)$ is
generated as a branching process whose conditional distribution given $N_c$
only depends on $t^c$ and is
entirely determined by the time varying fertility function $p(\cdot;\cdot)$.
This distribution is described through its conditional Laplace functional in
\cite[section~6.1]{roueff-vonsachs-sansonnet2016} under the additional
assumption
\begin{equation}
  \label{eq:NS-additional}
   \zzeta[\infty]<\infty   \quad\text{with}\quad \zzeta[q] =
   \sup_{t\in\mathbb{R}}\ellnorm{p(\cdot;t)}[q]\text{ for all $q\in[1,\infty]$}\;.
\end{equation}
The following result directly follows from these derivations.
A detailed proof is available in Appendix~\ref{sec:proof-prop-moment-bound} for sake of completeness.
\begin{proposition}\label{prop:uniform-moment-bound}
  Suppose that~(\ref{eq:cond-density-nonstat}) and~(\ref{eq:NS-additional}) hold and set
  \begin{equation}
    \label{eq:r1def-exp-bound}
r_1=\frac{-\log\zzeta}{2(1-\zzeta^{1/2}+\zzeta[\infty]\zzeta^{-1/2})}\;.
  \end{equation}
Then, for all $h:\rset\to\rset$ satisfying $\ellnorm{h}\leq1$ and
  $\ellnorm{h}[\infty]\leq1$,
  \begin{equation}
    \label{eq:exp-bound}
\esp\left[\rme^{(1-\zzeta^{1/2})r_1 |N(h)|}\right] \leq \exp\left(\ellnorm{\lambda_c}[\infty]
\,\zzeta^{-1/2}\,r_1\right) \;.
\end{equation}
Consequently, for all $q>0$, there exists a positive constant $B_q$ only
depending on $\ellnorm{\lambda_c}[\infty]$, $\zzeta$ and $\zzeta[\infty]$ such
that, for all $h:\rset\to\rset$ satisfying $\ellnorm{h}\leq1$ and
  $\ellnorm{h}[\infty]\leq1$,
  \begin{equation}
    \label{eq:moment-bound}
\left\|N(h)\right\|_q\leq B_q\;.
\end{equation}
\end{proposition}
The moment bound~(\ref{eq:moment-bound}) will be useful but far from sufficient
to bound the variance of our estimators efficiently. To see why, let us suppose
temporarily that $N$ is a homogeneous Poisson process with unit rate and
consider $h=n^{-1}\1_{[0,n]}-1$ for some positive integer $n$. Then the above
bound with $p=2$ gives $\var(N(h))\leq B_2^2$ although we know that in this very
special case we have $\var(N(h))=n^{-1}$, hence the bound is not sharp at all
as $n\to\infty$. In the following we provide new deviation bounds applying to
non-stationary Hawkes processes which allows one to recover the expected order
of magnitude for the variance. To this end we rely on stronger conditions than the
ones used in \cite{roueff-vonsachs-sansonnet2016}.

Define moreover, using the exponentially weighted norm notation
in~(\ref{eq:exp-weight-norm}), for all $d\geq0$, and $q\in[1,\infty]$,
$$
\zetaexp[q]{d}=\sup_{t\in\rset}\ellexpnorm{p(\cdot;t)}[q]\;.
$$
We strengthen the basic
conditions~(\ref{eq:locally-stat-hawk-finite-intensity})
and~(\ref{eq:NS-additional}) into the following one.
\begin{assumption}{NS}
\item \label{ass:non-stat-zeta-exp} We have
  $\ellnorm{\lambda_c}[\infty]<\infty$. Moreover,
for all $t\in\rset$, $p(\cdot;t)$ is supported on $\rset_+$ and
there exists $d>0$ such that $\zetaexp{d}<1$ and $\zetaexp[\infty]{d}<\infty$.
\end{assumption}
The version of~\ref{ass:non-stat-zeta-exp} in the locally
stationary case is \ref{ass:locally-stat-hawk-finite-expbounds} in the sense that the locally stationary Hawkes process
$(N_T)_{T\geq1}$ satisfies \ref{ass:locally-stat-hawk-finite-expbounds} if and
only if, for all $T\geq1$, the non-stationary Hawkes process $N_T$ satisfied
\ref{ass:non-stat-zeta-exp}. Therefore all the results below relying on 
\ref{ass:non-stat-zeta-exp} apply to the locally stationary Hawkes processes
satisfying  \ref{ass:locally-stat-hawk-finite-expbounds}.
We recall that this assumption means that the local fertility functions satisfy some uniform
exponential decreasing.

\subsection{New deviation bounds}
\label{sec:new-deviation-bounds}
The deviations bounds are based on the following exponential bound control on
component point processes $N(\cdot|t^c)$ defined above. 
\begin{proposition}\label{prop:Delta-exp-moment}
  Suppose that~\ref{ass:non-stat-zeta-exp} holds for some $d>0$. Then we have,
  for all $a\in(0,d)$, for all $t^c\in\Rset$,
  \begin{equation}
    \label{eq:Delta-exp-moment-preparation}
\esp\left[\int_{[t,\infty)}\rme^{a (s-t)}\; N(\rmd s|t^c)\right]\leq C_a
\quad\text{with}\quad C_a:= 1+\frac{\zetaexp[\infty]{d}}{(d-a)(1-\zetaexp{d})}\;.
  \end{equation}
\end{proposition}
\begin{proof}
Let $g:\rset\to\rset_+$ and define, for all  $h:\rset\to\rset_+$,
$$
[\mathcal{E}_g(h)](s)=g(s)+\int h(u)p(u-s;u)\;\rmd u \;.
$$
Following \cite[Proposition~7 and Eq.~(34)]{roueff-vonsachs-sansonnet2016}, we
have that, for all $t\in\rset$,
\begin{equation}
  \label{eq:expectation-limit}
\esp\left[N(g|t^c)\right]=\lim_{n\to\infty}[\mathcal{E}^n_g(g)](t) \;,
\end{equation}
where $\mathcal{E}^n_g$ denotes the $n$-th self-composition of
$\mathcal{E}_g$. In the following, we take some $a\in(0,d)$ and set
$$
g_t(s)= \rme^{a (s-t)} \1_{[t,\infty)}(s)=g_0(s-t) \;.
$$
Observe that, using that $p(;u)$ is supported on $[0,\infty)$
under~\ref{ass:non-stat-zeta-exp}, for all  $h:\rset\to\rset_+$,
\begin{align*}
[\mathcal{E}_{g_t}(h)](t)
  &\leq g_t(t)+\int_{t}^\infty h(u)
  p(u-t;u)\;\rmd u   \\
  &\leq1+
\int_{t}^\infty h(u)\rme^{d (t-u)}
  \rme^{d (u-t)}p(u-t;u)\;\rmd u\\
&\leq1+\int_{t}^\infty h(u)\rme^{d (t-u)}
\zetaexp[\infty]{d} \;\rmd u\\
&= 1 +
\zetaexp[\infty]{d}\ellexpnorm{h(t+\cdot)}[1][-d] \;.
\end{align*}
Applying this to $h=\mathcal{E}_{g_t}^{n-1}(g_t)$ we get, for all $n\geq1$,
\begin{align}\label{eq:bound-using-d1}
[\mathcal{E}_{g_t}^n(g_t)](t)
\leq 1+
\zetaexp[\infty]{d}\ellexpnorm{[\mathcal{E}_{g_t}^{n-1}(g_t)](t+\cdot)}[1][-d]
\end{align}
Similarly, we also get that, for all  $h:\rset\to\rset_+$,
\begin{align*}
  \ellexpnorm{[\mathcal{E}_{g_t}(h)](t+\cdot)}[1][-d]
&\leq \ellexpnorm{g_0}[1][-d]+
\zetaexp{d}\ellexpnorm{h(t+\cdot)}[1][-d] \;.
\end{align*}
Observing that
$\ellexpnorm{g_0}[1][-d]=(d-a)^{-1}$ and iterating the last inequality, we
finally obtain that, for all $n\geq1$,
$$
\ellexpnorm{[\mathcal{E}_{g_t}^n(g_t)](t+\cdot)}[1][-d]\leq
\frac1{d-a}\left(1+\zetaexp{d}+\dots+\zetaexp{d}^n\right)\leq
  \frac1{(d-a)(1-\zetaexp{d})}\;.
$$
Inserting this bound in~(\ref{eq:bound-using-d1}) and
letting $n$ go to $\infty$ as in~(\ref{eq:expectation-limit}) with $g=g_t$, we get the claimed result.
\end{proof}

Define the past and future $\sigma$-fields at time $t$ respectively as
\begin{align*}
\mathcal{F}_t&=\sigma\left(N_c(A),N(B|t^c)\;:\; A\in\mathbb{B}((-\infty,t]),
B\in\mathbb{B}(\rset), t^c\leq t\right)  \\
&\supset\sigma\left(N(A)\;:\; A\in\mathbb{B}((-\infty,t])\right)
\end{align*}
and
$$
\mathcal{G}_t=\sigma\left(N(A)\;:\; A\in\mathbb{B}((t,-\infty))\right)\;.
$$
The following result provides a uniform exponential control on the time-dependence of $N$.
\begin{proposition}\label{prop:cond-exp-control}
  Suppose that~\ref{ass:non-stat-zeta-exp} holds for some $d>0$ and that
  $\lambda_c$ is bounded.  Let $p\in[1,\infty)$, $t< u$ and $Y$ be a centered
  $L^1$ $\mathcal{G}_u$-measurable random variable. Then, for all $a\in(0,d)$,
  for all $q\in(p,\infty]$, if $Y$ is $L^{q}$,
$$
\left\|\esp[Y|\mathcal{F}_t]\right\|_p\leq \|Y\|_{q} \,
\left(\ellnorm{\lambda_c}[\infty]\,C_a\,a^{-1}
\rme^{-a(u-t)}\right)^{1/p-1/q} \;,
$$
where $C_a$ is defined in~(\ref{eq:Delta-exp-moment-preparation}).
\end{proposition}
\begin{proof}
In the following, we denote by $t^c_k$ the points of the Poisson process
$N_c$, that is,
$$
N_c=\sum_k\delta_{t^c_k}\;.
$$
  Define
$$
\Delta_{t^c}=\inf\left\{\delta>0~:~N([t^c+\delta,\infty)|t^c)=0\right\}\;.
$$
In other words, $\Delta_{t^c}$ is the size of the cluster $N(\cdot|t^c)$, that is the
distance between the most left point (which is $t^c$) and most right point. Since
$\Delta_{t^c}$ is a point among the points of $N(\cdot|t^c)$), we have
$$
\rme^{a \Delta_{t^c}}\leq\int_{[t,\infty)}\rme^{a (s-t)}\; N(\rmd s|t^c) \;,
$$
and a direct
consequence of Proposition~\ref{prop:Delta-exp-moment} is that, for all $a\in(0,d)$,
\begin{equation}
  \label{eq:Delta_t_exp_moment}
  \esp[\rme^{a \Delta_{t^c}}]\leq C_a\;.
\end{equation}

Now let us define the position of the last point generated by all clusters
started before time $t$, namely,
$$
\overline{\Delta}_t=\sup\left\{t^c_k+\Delta_{t^c_k}~:~t_k^c\leq
    t\right\}\;.
$$
We observe that $\overline{\Delta}_t$ is $\mathcal{F}_t$-measurable. Moreover,
if $t< u$ and  $Y$ is a centered $L^1$ $\mathcal{G}_u$-measurable random
variable, then we have $\esp[ Y|\mathcal{F}_t]=0$ on
$\{\overline{\Delta}_t<u\}$. The Hölder inequality then yields for $1\leq p<q\leq\infty$
$$
\left\|\esp[Y|\mathcal{F}_t]\right\|_p=
\left\|\esp[Y|\mathcal{F}_t]\1_{\{\overline{\Delta}_t\geq u\}}\right\|_p
\leq \left\|\esp[Y|\mathcal{F}_t]\right\|_q\left(\pr(\overline{\Delta}_t\geq u)\right)^{1/p-1/q}\;.
$$
Since $\left\|\esp[Y|\mathcal{F}_t]\right\|_q\leq \|Y\|_q$, it only remains to
prove that
\begin{equation}
  \label{eq:remains-forgetting-exponential}
  \pr(\overline{\Delta}_t\geq u)\leq C_0\rme^{-\lambda_0(u-t)}\;.
\end{equation}
Observe that $M=\sum_{k}\delta_{t^c_k,\Delta_{t^c_k}}$ is a marked Poisson
point process such that, given $N_c$, the marks $\Delta_{t^c_k}$ are
independent and for each $k$ the conditional distribution of $\Delta_{t^c_k}$
only depends on $t_k^c$. Hence $M$ is a Poisson point process with points
valued in $\rset\times\rset_+$ and
$$
\{\overline{\Delta}_t\geq u\} =
\left\{M(\{(t^c,\delta)\in(-\infty,t]\times\rset_+~:~t_c+\delta\geq u\})>0\right\}\;.
$$
We thus get that
$$
  \pr(\overline{\Delta}_t\geq u)=1-\exp\left(-\int_{-\infty}^{t} \pr(t_c+\Delta_{t^c}\geq
    u)\lambda_c(t^c)\;\rmd t^c\right)\leq
\ellnorm{\lambda_c}[\infty]\int_{-\infty}^{t}
\pr(\Delta_{t^c}\geq u-t_c)\;\rmd t^c\;,
$$
where we used that $1-\rme^{-x}\leq x$ for all $x\geq0$.
Using~(\ref{eq:Delta_t_exp_moment}) and the exponential Markov inequality, it follows that
$$
\pr(\overline{\Delta}_t\geq u)\leq
\ellnorm{\lambda_c}[\infty]\,C_a\,
\int_{-\infty}^{t} \rme^{a(t^c-u)}\;\rmd t^c
=\ellnorm{\lambda_c}[\infty]\,C_a\,
a^{-1}\rme^{a(t-u)}\;.
$$
\end{proof}

We can now derive a Burkhölder-type inequality.

\begin{proposition}\label{prop:burkh-ineq}
  Suppose that~\ref{ass:non-stat-zeta-exp} holds for some $d>0$.  Let
  $p\in[2,\infty)$. Then there exists a positive constant $B_p$ such that, for
  all bounded functions $h$ with support included in $[j,j+n]$ for some
  $j\in\zset$ and $n\in\nset$,
$$
\left\|N(h)-\esp[N(h)]\right\|_p\leq A\;\ellnorm{h}[\infty] \;\sqrt{n}\;.
$$
where $A$ is a positive constant only depending on $d$,
$\ellnorm{\lambda_c}[\infty]$, $\zzeta$, $\zzeta[\infty]$,
$\zetaexp[\infty]{d}$ and $\zetaexp{d}$, e.g., for any $a\in(0,d)$ and $q>p$,
$$
A:=(B_1+B_p)(B_1+B_q)\left(\ellnorm{\lambda_c}[\infty]\,C_a\,a^{-1}\right)^{1/p-1/q}
\frac{\rme^{-a(1/p-1/q)}}{1-\rme^{-a(1/p-1/q)}}\;,
$$
where $B_p$ is defined in Proposition~\ref{prop:uniform-moment-bound} and $C_a$ in~(\ref{eq:Delta-exp-moment-preparation}).
\end{proposition}
\begin{proof}
We can assume $\ellnorm{h}[\infty]=1$ without loss of generality. We write
$$
h=\sum_{i=1}^n h_i\quad\text{with}\quad h_i=h\1_{[j+i-1,j+i)}\;.
$$
Then $\ellnorm{h_i}[\infty]\leq1$ and $\ellnorm{h_i}\leq1$ for all $i$ and,
defining $X_i=N(h_i)-\esp[N(h_i)]$, from
Proposition~\ref{prop:uniform-moment-bound}, we have, for all $q\geq 1$,
\begin{equation}
  \label{eq:moment-bound-Xi}
  \|X_i\|_q\leq B_q+B_1 \;.
\end{equation}
Then $N(h)-\esp[N(h)]=\sum_{i=1}^n X_i$
and, applying \cite[Proposition~5.4, Page~123]{dedecker07},
we have
\begin{equation}
  \label{eq:burk-Xi}
\left\|N(h)-\esp[N(h)]\right\|_p\leq\left(2p\sum_{i=1}^{n}b_{i,n}\right)^{1/2}\;,
\end{equation}
where, denoting $\mathcal{M}_i=\mathcal{F}_{j+i}$,
$$
b_{i,n}=\max_{1\leq \ell\leq
  n}\left\|X_i\sum_{k=i}^\ell\esp[X_k|\mathcal{M}_i]\right\|_{p/2}\;.
$$
Observing that $X_k$ is centered and $\mathcal{G}_{j+k-1}$--measurable,
Proposition~\ref{prop:cond-exp-control}, gives that, for any $q>p$,
$$
\left\|\esp[X_k|\mathcal{M}_i]\right\|_{p}\leq \|X_k\|_q
\left(\ellnorm{\lambda_c}[\infty]\,C_a\,a^{-1}
\rme^{-a(k-i+1)}\right)^{1/p-1/q}\;.
$$
The Hölder inequality, the last two displays and~(\ref{eq:moment-bound-Xi})
yield, for all $q>p$,
$$
b_{i,n}\leq
(B_1+B_p)(B_1+B_q)\left(\ellnorm{\lambda_c}[\infty]\,C_a\,a^{-1}\right)^{1/p-1/q}
\frac{\rme^{-a(1/p-1/q)}}{1-\rme^{-a(1/p-1/q)}}\;.
$$
Applying this in~(\ref{eq:burk-Xi}), we get the claimed bound.
\end{proof}

Another consequence of Proposition~\ref{prop:cond-exp-control} is the
following useful covariance bound.
\begin{corollary}
  \label{cor:cov-squares}
  Suppose that~\ref{ass:non-stat-zeta-exp} holds for some $d>0$ and that
  $\lambda_c$ is bounded.  Let $h_1$ and $h_2$ be two bounded integrable
  functions. Let $\gamma$ satisfy one of the following assertions.
  \begin{enumerate}[label=(\roman*)]
  \item\label{item:gamma-ass1} There exist $t\in\rset$ such that $\supp(h_1)\subset(-\infty,t]$ and $\supp(h_2)\subset[t+\gamma,\infty)]$.
  \item\label{item:gamma-ass2} There exist $t\in\rset$ such that $\supp(h_2)\subset(-\infty,t]$ and $\supp(h_1)\subset[t+\gamma,\infty)]$.
  \item\label{item:gamma-ass3} $\gamma=0$.
  \end{enumerate}
Then for all $q>4$, there exists $C_q>0$ and $\alpha_q>0$ both only depending on
$q$, $\ellnorm{\lambda_c}[\infty]$, and $a$ and $C_a$ in Proposition~\ref{prop:cond-exp-control}
such that
\begin{equation}
  \label{eq:cov-square-bound}
\left| \cov\left(\left|\overline{\overline{N}}(h_1)\right|^2,\left|\overline{\overline{N}}(h_2)\right|^2\right)\right|
\leq C_q  \left\|\overline{\overline{N}}(h_1)\right\|_q^2\;
\left\|\overline{\overline{N}}(h_2)\right\|_q^2
\;\rme^{-\alpha_q \gamma} \;,
\end{equation}
where $\overline{\overline{N}}(h)=N(h)-\esp[N(h)]$.
\end{corollary}

\begin{proof}
  In the case~\ref{item:gamma-ass3}, the bound~(\ref{eq:cov-square-bound})
  actually holds with $q=4$ by the 
  Cauchy--Schwarz inequality and thus also holds with $q>4$ by Jensen's
  inequality. 

We now consider the case~\ref{item:gamma-ass1} (the last
one~\ref{item:gamma-ass2} being obtained by inverting $h_1$ and $h_2$). We have
in this case, denoting $Y=\left|\overline{\overline{N}}(h_2)\right|^2-\esp\left[\left|\overline{\overline{N}}(h_2)\right|^2\right]$,
\begin{align*}
\left| \cov\left(\left|\overline{\overline{N}}(h_1)\right|^2,\left|\overline{\overline{N}}(h_2)\right|^2\right)\right|
&=\esp\left[\left|\overline{\overline{N}}(h_1)\right|^2\left(\esp[Y|\mathcal{F}_t]\right)\right] \\
&\leq
\left\|\overline{\overline{N}}(h_1)\right\|_4^2\;\left\|\esp[Y|\mathcal{F}_t]\right\|_4^2\;.
\end{align*}
The Jensen Inequality and $q>4$ give that
$\left\|\overline{\overline{N}}(h_1)\right\|_4\leq\left\|\overline{\overline{N}}(h_1)\right\|_q$ and
the proof is concluded by using   Proposition~\ref{prop:cond-exp-control} to
bound $\left\|\esp[Y|\mathcal{F}_t]\right\|_4$. 
\end{proof}

\section*{Acknowledgements}

We thank Adil Reghai of Natixis for providing us with the ESSI.PA and TOTF.PA
transactions data sets.
Rainer von Sachs gratefully acknowledges funding by contract ``Projet d'Actions
de Recherche Concert\' ees'' No. 12/17-045 of the ``Communaut\'e fran\c caise
de Belgique'' and by IAP research network Grant P7/06 of the Belgian government
(Belgian Science Policy). Both authors thank L. Sansonnet for her contributing discussions. 

\appendix

\section{Useful lemmas}
In the following lemmas, we have gathered some simple bounds that we will
repeatedly use in the sequel.
\begin{lemma}
  \label{lem:norm-estimates}
  Let $a,\beta>0$ and  $p\in[1,\infty]$.
For any function $g$ and any $t\in\rset$, we have 
\begin{align}
  \label{eq:betnorm-shifted-bound}
&\ellbetanorm{g(\cdot-t)}[p]\leq 2^{(\beta-1)_+}\,\left(\ellbetanorm{g}[p]+|t|^\beta\ellnorm{g}[p]\right)\; ,\\
  \label{eq:a1norm-shifted-bound}
&\ellexpnorm{g(\cdot-t)}[1][a]\leq\ \rme^{a|t|}\ \ellexpnorm{g}[1][a]\; .
\end{align}
Let $\kbofreq=\bfreq^{1/2}
\rme^{\rmi\omega_0t}\kfreq(\bfreq t)$ and
$\ktbtime=(T\btime)^{-1}\,\ktime(u/(T\btime))$, where 
the kernels $\ktime$ and $\kfreq$ satisfy
$\ellnorm{\ktime}=\ellnorm{\kfreq}[2]=1$. 
Then, we have, for all $\btime,\bfreq\in(0,1]$ and $T>0$, 
 \begin{align}
 \label{eq:K-betanorm}
& \ellbetanorm{\kbofreq} \ =\ \bfreq^{-1/2-\beta}  \ellbetanorm{\kfreq} \;,   \\
  \label{eq:K-onenorm}
& \ellnorm{\kbofreq}\  =\ \bfreq^{-1/2}  \ellnorm{\kfreq}\;,   \\
 \label{eq:K-twonorm}
& \ellnorm{\kbofreq}[2]\  =\ 1 \;,   \\
  \label{eq:K-inftynorm}
& \ellnorm{\kbofreq}[\infty]\  =\  \bfreq^{1/2} \ellnorm{\kfreq}[\infty]\;,   \\
 \label{eq:K-2betanorm}
& \ellbetanorm{\kbofreq}[2]\  =\ \bfreq^{-\beta}  \ellbetanorm{\kfreq}[2]\;,   \\ 
  \label{eq:W-onenorm}
& \ellnorm{\ktbtime}\  = 1 \ \;,   \\
  \label{eq:W-inftynorm}
& \ellnorm{\ktbtime}[\infty]\  =\  (\btime T)^{-1} \ellnorm{\ktime}[\infty]\;,   \\ 
 \label{eq:W-betanorm}
& \ellbetanorm{\ktbtime} \ =\  (\btime T)^{\beta} \ellbetanorm{\ktime}\ \;,   \\
  \label{eq:g-onenorm}
& \ellnorm{\kbofreq \ast \ktbtime}\  \leq\ \bfreq^{-1/2}  \ellnorm{\kfreq}\ \;,   \\
  \label{eq:g-twonorm}
& \ellnorm{\kbofreq \ast \ktbtime}[2]\  \leq\ \ellnorm{\ktbtime}[2] \ellnorm{\kbofreq} \leq\ \bfreq^{-1/2} (\btime T)^{-1/2} \ellnorm{\ktime}[2] \ellnorm{\kfreq}\;,   \\
  \label{eq:g-inftynorm}
& \ellnorm{\kbofreq \ast \ktbtime}[\infty]\   \leq\  (\btime T)^{-1}\bfreq^{-1/2} \ellnorm{\ktime}[\infty] \ellnorm{\kfreq}\;,   \\ 
 \label{eq:g-betanorm}
& \ellbetanorm{\kbofreq \ast \ktbtime} \ \leq
  2^{(\beta-1)_+}\, \left(\bfreq^{-1/2}\ (\btime T)^\beta\ \ellnorm{\kfreq} \ellbetanorm{\ktime}\
 +  \bfreq^{-1/2-\beta} \ellbetanorm{\kfreq}\right)\;.
 \end{align}
If $\kfreq$ and $\ktime$ have compact supports both included in $[-\tilde a,\tilde
a]$ for some $\tilde a>0$, we have
 \begin{align}
 \label{eq:Ktime-a1norm}
   &\ellexpnorm{\ktbtime}[1][a]  =\ \ellexpnorm{\ktime}[1][aT\btime]
  \leq\ \rme^{a\,\tilde{a}T\btime}
  \ellnorm{\ktime} \;,   \\
 \label{eq:K-a1norm}
& \ellexpnorm{\kbofreq}[1][a]\  =\ \bfreq^{-1/2}  \ellexpnorm{\kfreq}[1][a/\bfreq]
  \leq\ \bfreq^{-1/2} \rme^{a\,\tilde{a}\bfreq^{-1}}
  \ellnorm{\kfreq} \;,   \\
   \label{eq:g-a1norm}
& \ellexpnorm{\kbofreq \ast \ktbtime}[1][a] \ \leq\  \ellexpnorm{\ktbtime}[1][a]\ \ellexpnorm{\kbofreq}[1][a]\
\leq\  \bfreq^{-1/2} \rme^{a\,\tilde{a}(T\btime+\bfreq^{-1})}  \ellnorm{\kfreq}\;.
 \end{align}
\end{lemma}
\begin{proof}
  All these bounds are straightforward. We use the usual $L^p$ bounds for
  convolution $\ellnorm{h\star g}\leq \ellnorm{h}\ellnorm{g}$ and
  $\ellnorm{h\star g}[2]\leq \ellnorm{h}[2]\ellnorm{g}$. When necessary, the
  weights are handled by using
  \begin{equation}
    \label{eq:weight-trick-vonvolution}
  |s|^\beta\leq2^{(\beta-1)_+}(|s-t|^\beta+|t|^{\beta})\quad\text{and}\quad
  \rme^{a|s|}\leq\rme^{a|t|}\rme^{a|s-t|}\;.   
  \end{equation} 
\end{proof}
\begin{lemma}
\label{lemma-expo-neglig}  Let $T,\btime,\bfreq$
  satisfy (\ref{eq:conditions-T-and-co-timebias-bartlett})
  and~(\ref{eq:conditions-T-and-co-timebias-bartlett-expoterms}). Then for all
  $a_1,a_2>0$, we have
  $$\exp(a_2(T\btime+\bfreq^{-1})-a_1T)\leq
  \max\left(1,\exp\left(2a_2\rme^{1/a_1}\right)\right)\;.$$
  In particular, under Convetion~\ref{conv:As}, we
  have $\exp(A_2(T\btime+\bfreq^{-1})-A_1T)\lesssim1$.
\end{lemma}
\begin{proof}
  First note that, by~(\ref{eq:conditions-T-and-co-timebias-bartlett}),
  $\bfreq^{-1}\leq T\btime$ and thus
  $$
  \exp(a_2(T\btime+\bfreq^{-1})-a_1T)\leq\exp(2a_2T(\btime-a_1))\;.
  $$
  If $\btime\leq a_1$ this upper bound is at most 1 and otherwise, using that
  $\btime\leq1$ and
  then~(\ref{eq:conditions-T-and-co-timebias-bartlett-expoterms}), we have that
  $T(\btime-a_1)\leq T\leq \rme^{1/\btime}\leq\rme^{1/a_1}$.
\end{proof}
\section{Proof of Theorem~\ref{thm:new-var-approx}}
\label{sec:proof_Thm4}
\subsection{A useful lemma}
The following lemma prepares the ground for deriving appropriate bounds used in
the proof of Theorem~\ref{thm:new-var-approx}.
\begin{lemma} 
\label{lem:Prep_Proof_Thm4}
Let $\varphi: \Rset^2 \to \Rset_+$ and $f: \Rset \to \Rset_+$ with $f \in L^1
\cap L^2$. Let moreover $f_\infty: \Rset \to \Rset_+$ satisfying, for all $s \in \Rset$,
\begin{equation}
\label{eq:f_infty_bound}
f_\infty(s) \ \leq\ f(s)\ +\ \int f_\infty(t)\ \varphi(s-t; t)\ \rmd t\;.
\end{equation}
Let us consider the following conditions depending on some $a\geq0$ and $M\in(0,\infty]$.
\begin{enumerate}[label=(C-\arabic*)]
\item \label{item:zeta1-a--lemma}
$\displaystyle\zetaexp{a} := \sup_t \int \varphi(u;t)\ \rme^{a|u|}\ \rmd u < 1$.
\item\label{item:zetaoo-a--lemma} $\displaystyle\zetaexp[\infty]{a} := \sup_{u,t
  }\varphi(u;t)\rme^{a|u|} < \infty$.
\item\label{item:foo_-M--lemma} 
$\displaystyle\overline{\varphi}(u) := \sup_{|t| \leq M} \varphi(u;t)$
satisfies
$\displaystyle\tilde \zzeta:= \int \overline{\varphi}(u)\ \rmd u < 1$.
\end{enumerate}
If $M < \infty$, we define $\overline{f}_\infty:\rset\to\rset_+$ by
$\displaystyle
\overline{f}_\infty(s) := \int_{|t| >M} f_\infty(t)\ \varphi(s-t;t)\ \rmd t\ .$
Then the following assertions hold.
\begin{enumerate}[label=(\roman*)]
\item\label{item:Prop_tech_part3} Condition~\ref{item:zeta1-a--lemma} with $a\geq0$ implies
$\displaystyle \ellexpnorm{f_{\infty}}[1][a]\lesssim\ellexpnorm{f}[1][a]$.
\item\label{item:Prop_tech_part2}
 Condition~\ref{item:zeta1-a--lemma} with $a>0$ implies, for any
 $\beta\in(0,1]$, $\ellbetanorm{f_\infty}\lesssim\ellbetanorm{f} +
    \ellnorm{f}$.
\item\label{item:Prop_tech_part7prim}
 Condition~\ref{item:foo_-M--lemma} with $M=\infty$ implies
 $\displaystyle \ellnorm{f_\infty}[2]\lesssim\ellnorm{f}[2]$.
\item  \label{item:Prop_tech_part8prim}
 Conditions~\ref{item:zetaoo-a--lemma}  with
 $a>0$ and~\ref{item:foo_-M--lemma} with $M=\infty$ imply
 $\ellbetanorm{f_\infty}[2]\lesssim\ellbetanorm{f}[2] +\ellnorm{f}[2]$.
\item \label{item:Prop_tech_part7}
 Conditions~\ref{item:zeta1-a--lemma} with $a>0$,
 \ref{item:zetaoo-a--lemma} with $a=0$ and~\ref{item:foo_-M--lemma} with
  $M<\infty$ imply $\ellnorm{f_\infty}[2]\lesssim\ellnorm{f}[2]+ \rme^{-a\,M}\ellexpnorm{f}[1][a]$.
\item \label{item:Prop_tech_part8}
 Conditions~\ref{item:zeta1-a--lemma}
 and \ref{item:zetaoo-a--lemma} with $a>0$ and~\ref{item:foo_-M--lemma} with
  $M<\infty$ imply, for any
 $\beta\in(0,1]$, $\ellbetanorm{f_\infty}[2]\lesssim\ellbetanorm{f}[2] +  \rme^{-aM/2}  \ellexpnorm{f}[1][a]+    \ellnorm{f}[2]$.
\end{enumerate}
Here ``$\lesssim\dots$'' means ``$\leq C\ \dots$'' with a positive constant $C$ 
possibly depending on $a$, $\zetaexp[1]{a}$, $\zetaexp[\infty]{a}$, $\tilde \zzeta$ or
$\beta$ only (thus neither depending on $M$ nor on $f$).
\end{lemma}

\begin{proof}
Using~(\ref{eq:weight-trick-vonvolution}) to deal with weighted $L^p$-norms,
the bound~(\ref{eq:f_infty_bound}) easily yields, for all $a\geq0$ and $\beta\in(0,1]$,
\begin{align}\label{eq:f-f-infty-weights-norms-bounds-1}
\ellexpnorm{f_{\infty}}[1][a]&  \leq \ellexpnorm{f}[1][a] + \zetaexp{a}\ellexpnorm{f_{\infty}}[1][a]\;,\\
  \label{eq:f-f-infty-weights-norms-bounds-infty}
  \ellbetanorm{f_\infty} &\leq \ellbetanorm{f} +  \left(
\sup_t\ellbetanorm{\varphi(\cdot;t)}\right) \ellnorm{f_\infty} +\zetaexp{0}\ellbetanorm{f_\infty} \;. 
\end{align}
The bound~(\ref{eq:f-f-infty-weights-norms-bounds-1})
yields~\ref{item:Prop_tech_part3}.
Moreover, using, for any $\beta\in(0,1]$,  $C_{a,\beta} := \sup_{x \geq 0} x^\beta\ \rme^{-ax} \leq \sup_{x \geq 0}
x\ \rme^{-ax} \leq (\rme a)^{-1}$, we get that
$$
\sup_t\ellbetanorm{\varphi(\cdot;t)}=\sup_t\int\varphi(r; t)\ |r|^\beta\ \rmd r \leq\zetaexp{a}(\rme a)^{-1}\;.
$$
Using this in~(\ref{eq:f-f-infty-weights-norms-bounds-infty})
and~\ref{item:Prop_tech_part3} with $a=0$, we get~\ref{item:Prop_tech_part2}.

\noindent The definition of $\overline{\varphi}$ in the case $M=\infty$ allows
us to bound the second term of the
bound in~(\ref{eq:f_infty_bound}) by
\begin{equation}
  \label{eq:bound-case-M-infty-convolution}
\int f_\infty(t)\ \varphi(s-t; t)\ \rmd t \leq f_\infty\ast\overline{\varphi}(s)\;.  
\end{equation}
Using this, we get in turn that
$$
\ellnorm{f_\infty}[2]\ \leq\ \ellnorm{f}[2]+\ellnorm{f_\infty\ast\overline{\varphi}}[2]
\leq \ellnorm{f}[2]+\ellnorm{f_\infty}[2]\ellnorm{\overline{\varphi}}\;,
$$
which under
Condition~\ref{item:foo_-M--lemma} yields~\ref{item:Prop_tech_part7prim}.

\noindent Similarly to~(\ref{eq:f-f-infty-weights-norms-bounds-infty}) and with
the definition of $\overline{\varphi}$ in the case $M=\infty$, we have
$$
\ellbetanorm{f_\infty}[2] \leq \ellbetanorm{f}[2] +
\ellbetanorm{\overline{\varphi}}\ellnorm{f_\infty}[2]
+\ellnorm{\overline{\varphi}}\ellbetanorm{f_\infty}[2] \;. 
$$
Note that in the case $M=\infty$, $\ellbetanorm{\overline{\varphi}}\lesssim1$ as
a consequence of~\ref{item:zetaoo-a--lemma} with $a>0$.
Thus, using~\ref{item:Prop_tech_part7prim} to bound $\ellnorm{f_\infty}[2]$, under
Condition~\ref{item:foo_-M--lemma}, we get~\ref{item:Prop_tech_part8prim}.

\noindent Assertions~\ref{item:Prop_tech_part7}
and~\ref{item:Prop_tech_part8} are obtained similarly
as~\ref{item:Prop_tech_part7prim} and~\ref{item:Prop_tech_part8prim} but with
an additional step to deal with a finite $M$. Namely we have in this case
$$
\int f_\infty(t)\ \varphi(s-t; t)\ \rmd t \leq  f_\infty\ast\overline{\varphi}(s)+\overline{f}_\infty(s)\;.
$$
It follows that the same bounds as in~\ref{item:Prop_tech_part7prim}
and~\ref{item:Prop_tech_part8prim} applies but with $f$ replaced by
$f+\overline{f}_\infty$. Hence to obtain the bounds in~\ref{item:Prop_tech_part7}
and~\ref{item:Prop_tech_part8}, we only need to show that, under the
corresponding conditions, we have
\begin{align}
  \label{eq:toshow-7from7prim}
  \ellnorm{\overline{f}_\infty}[2] &\lesssim\rme^{-a\,M}\ellexpnorm{f}[1][a]\;, \\
  \label{eq:toshow-8from8prim}
  \ellbetanorm{\overline{f}_\infty}[2]&\lesssim\rme^{-a\,M/2}\ellexpnorm{f}[1][a]\;.
\end{align}
Observe that, by definition of $\overline{f}_\infty$, we have
\begin{align*}
&  \ellnorm{\overline{f}_\infty} = \int_{|t| >M} f_\infty(t)
  \left(\int \varphi(s-t;t)\ \rmd s\right) \rmd t\ \leq\ \zetaexp{0} \int_{|t| >M} f_\infty(t)\ \rmd t \ ,\\
& \ellnorm{\overline{f}_\infty}[\infty]=\sup_s \int_{|t| >M} f_\infty(t)\
  \varphi(s-t;t)\ \rmd t\ \leq \zetaexp[\infty]{0} \int_{|t| >M} f_\infty(t)\
  \rmd t\;.
\end{align*}
Using $\ellnorm{\overline{f}_\infty}[2]\leq \left( \ellnorm{\overline{f}_\infty}[\infty]
  \ellnorm{\overline{f}_\infty}\right)^{1/2}$,
$\displaystyle
\int_{|t| >M} f_\infty(t)\
  \rmd t
\ \leq\ \rme^{-aM}\ \ellexpnorm{f_\infty}[1][a]$
and the bound in~\ref{item:Prop_tech_part3} to bound
$\ellexpnorm{f_\infty}[1][a]$ with $\ellexpnorm{f}[1][a]$,
 we get~(\ref{eq:toshow-7from7prim}).

\noindent Finally, we prove~(\ref{eq:toshow-8from8prim}). First we note that,
using~(\ref{eq:weight-trick-vonvolution}), 
for $q=1,\infty$, we have
$$
\ellbetanorm{\overline{f}_\infty}[q]\leq \left(\sup_t\ellbetanorm{\varphi(\cdot;t)}[q]\right)
\left(\int_{|t| >M} f_\infty(t)\
  \rmd t\right)+
\left(\sup_t\ellnorm{\varphi(\cdot;t)}[q]\right) \left(\int_{|t| >M} f_\infty(t)\ |t|^\beta\ 
  \rmd t \right)\;.
  $$
  The bound~(\ref{eq:toshow-8from8prim}) then follows similarly
  as~(\ref{eq:toshow-7from7prim}) by using
  $\ellbetanorm{\overline{f}_\infty}[2]\leq \left( \ellbetanorm{\overline{f}_\infty}[\infty]
  \ellbetanorm{\overline{f}_\infty}\right)^{1/2}$ and (with the constant
$C_{a/2,\beta}$ defined as above)
$$
\int_{|t| >M} f_\infty(t)\ |t|^\beta\ 
  \rmd t \leq C_{a/2,\beta}\rme^{-a\,M/2}\ellexpnorm{f_\infty}[1][a]\leq2(\rme a)^{-1}\rme^{-a\,M/2}\ellexpnorm{f_\infty}[1][a]\;.
$$ 
This concludes the proof.
\end{proof}

\subsection{Preliminaries}
\label{sec:preliminaries}
In these preliminaries, we only require~(\ref{eq:cond-density-nonstat})
and~(\ref{eq:NS-additional}) to hold as they are sufficient to define a
non-stationary Hawkes
process. Under~\ref{ass:locally-stat-hawk-finite-expbounds}, for any given
$T\geq1$, these conditions are satisfied by the parameters of the
non-pstationary Hawkes process $N=N_T$, with immigrant intensity
$\lambda_{cT}(t)=\locstat{\lambda}_c(t/T)$ and varying fertility function
$p_T(u;t)=\locstat{p}(u;t/T)$.

Let $g \in L^1 \cap L^\infty$, hence $g \in L^2$, too.  Since $N_c$ is a
Poisson point process with intensity $\lambda_c$ and the clusters
$N(\cdot|t^c)$ can be seen as conditionally independent marks of this Poisson
process, we have
\begin{equation}
  \label{eq:var-condvar-condexp}
  \var(N(g)) \ = \ \ \int \var\left(N(g|t^c)\right)\ \lambda_c(t^c)\ \rmd t^c
  + \int \left(\esp[N(g|t^c)]\right)^2\ \lambda_c(t^c)\ \rmd t^c
\end{equation}
From~[Corollary~8 and
Proposition~11]\cite{roueff-vonsachs-sansonnet2016} applied to the function
$(s,z) \mapsto s\ g(z)$, we obtain that the applications
$t \mapsto \esp[N(g|t)]$ and $t \mapsto \var(N(g|t))$ are fixed points of two
$L^1\to L^1$ operators $h \mapsto \mce(h)$ and $h \mapsto \mce[\tilde g](h)$
defined as follows. The operator $\mce$ is defined by
\begin{equation}
\label{eq:new_fixed_point} \forall\ s \in \Rset, \quad \mce(h)(s)\ =\ g(s)\
+\ \int h(t)\ p(t-s; t)\ \rmd t\; .    
\end{equation}
 And the operator $\mce[\tilde g]$ is
defined similarly but with $g$ replaced by the function
$$
\tilde{g}(s)\ =\  \int (\esp[N(g|t)])^2\ p(t-s; t)\ \rmd t\; .
$$
By the first condition in~(\ref{eq:cond-density-nonstat}), $\mce$ and
$\mce[\tilde g]$ are strictly contracting and thus admit a unique fixed
point.  We denote this fixed point by $\mcei$ (resp. $\mcei[\tilde{g}]$) in the following.

Hence, to summarize, the computation of $\var(N(g))$ boils down to the formula
\begin{equation}
  \label{eq:var-with-fiwed-points}
  \var(N(g)) \ = \ \ \int \mcei[\tilde g](t) \lambda_c(t^c)\ \rmd t^c
  + \int \left(\mcei(t)\right)^2\ \lambda_c(t^c)\ \rmd t^c\;,
\end{equation}
where $\mcei$ and $\mcei[\tilde g]$ are the unique fixed points of the
$L^1\to L^1$ operators $\mce$ and $\mce[\tilde g]$, with $\mce$ defined
by~(\ref{eq:new_fixed_point}) and $\mce[\tilde g]$ defined similarly but with
$g(s)$ replaced by
\begin{equation}
  \label{eq:tide-g-def}
\tilde{g}(s)\ =\  \int (\mcei(t))^2\ p(t-s; t)\ \rmd t\; .
\end{equation}

\subsection{Decomposition of the approximation}
The framework introduced in Section~\ref{sec:preliminaries} applies under the
assumptions of Theorem~\ref{thm:new-var-approx} for computing both
$\var\left(N_T(S^{-Tu}g))\right)$ and $\var\left(N(g;u)\right)$.  For
simplicity and without loss of generality we take $u=0$ in the following and
thus wish to approximate $\var\left(N_T(g))\right)$ with
$\var\left(\stat{N}(g)\right)$, with $\stat{N}(g):=N(g;0)$.  Then
$\var\left(N_T(g))\right)$ and $\var\left(N(g;0)\right)$ satisfy
Eq.~(\ref{eq:var-with-fiwed-points}) by adapting the definitions of $\mce$,
$\mcei$, $\tilde{g}$ and $\mcei[\tilde g]$ to the corresponding $p(t-s; t)$ and
$\lambda_c$. Namely, to compute $\var\left(N_T(g))\right)$, we apply these
equations and definitions with $p(t-s; t) = \locstat{p}(t-s;t/T)$ and
$\lambda_c(t^c)=\locstat{\lambda}_c(t^c/T)$, while to compute
$\var\left(N(g;0))\right)$, we apply them with
$p(t-s; t) = \stat{p}(t-s):=\locstat{p}(t-s;0)$ and
$\lambda_c(t^c)=\stat{\lambda}_c:=\locstat{\lambda}_c(0)$.  To distinguish
between these two cases, we use the corresponding notation $\mcet$, $\mceti$,
$\tilde{g}^{(T)}$ and $\mceti[\tilde g]$ in the first case and $\mces$,
$\mcesi$, $\stat{\tilde{g}}$ and $\mcesi[\tilde g]$ in the
second case.

\noindent Applying~(\ref{eq:var-with-fiwed-points}) then yields the bound
\begin{align}
  \label{eq:7}
  \left| \var(N_T(g)) \right| \ \leq\ 
  \ellnorm{\locstat{\lambda}_c}[\infty]\left(\ellnorm{\mceti[\tilde]}+\ellnorm{\mceti}[2]^2\right)\;,
\end{align}
and the approximation bound
\[
\left| \var(N_T(g)) - \var(\stat{N}(g)) \right| \ \leq\ (A)\ +\ (B)\ + \ (C)\ + \ (D)
\]
with
\begin{align*}
&(A)\ =\ \int \left(\mcesi(t^c)\right)^2\ \left|\locstat{\lambda}_c\left(\frac{t^c}{T}\right)-\stat{\lambda}_c\right|\ \rmd t^c\; ,\\
&(B)\ =\ \ellnorm{\locstat{\lambda}_c}[\infty]
  \int \left|\left(\mceti(t^c)\right)^2 - \left(\mcesi(t^c)\right)^2\ \right|\ \rmd t^c\; ,\\
&(C)\ =\ \int \mcesi[tilde](t^c)\ \left|\locstat{\lambda}_c\left(\frac{t^c}{T}\right)-\stat{\lambda}_c\right|\ \rmd t^c\; ,\\
&(D)\ =\ \ellnorm{\locstat{\lambda}_c}[\infty] \int \left|\mceti[tilde](t^c) -\mcesi[tilde](t^c) \right|\ \rmd t^c\; .
\end{align*}
Note that in (A) and $(C)$, using \ref{ass:locally-stat-hawk-center-intensity-lip} we can bound 
$|\locstat{\lambda}_c(t^c/T)-\stat{\lambda}_c |=|\locstat{\lambda}_c(t^c/T)-\locstat{\lambda}_c(0) |
\ \leq\   \holderlsc\ T^{-\beta}\ |t^c|^\beta$.
Hence, these four terms can be bounded using the previously introduced weighted
norms, and we get
\begin{multline}
  \label{eq:var-approx-allweightednorms}
\left| \var(N_T(g)) - \var(\stat{N}(g)) \right|
\lesssim T^{-\beta} \left(\ellbetanorm{\mcesi[g]}[2][\beta/2]^2+\ellbetanorm{\mcesi[\tilde g]}\right)\\+
  \ellnorm{\left(\mceti[g]\right)^2 - \left(\mcesi[g]\right)^2}+ \ellnorm{\mceti[\tilde g] - \mcesi[\tilde g]}\;.
\end{multline}
To bound these norms, we successively apply Lemma~\ref{lem:Prep_Proof_Thm4} in various settings.
\subsection{Successive applications of Lemma~\ref{lem:Prep_Proof_Thm4}}
\paragraph{Norms involving $\mcesi[g]$~:}
We apply Lemma~\ref{lem:Prep_Proof_Thm4} with
$f=|g|$, $\varphi(u;t)=\overline{\varphi}(u)= \stat{p}(-u)= \locstat{p}(-u;0)$
and $f_\infty = \left|\mcesi[g]\right|$. In
this setting Eq.~(\ref{eq:f_infty_bound}) is inherited from the fact that
$\mcesi$ is a fixed point of $\mces$. Conditions~\ref{item:zeta1-a--lemma}, \ref{item:zetaoo-a--lemma}
and~\ref{item:foo_-M--lemma} hold with $a=d>0$ and $M=\infty$ as consequences
of~\ref{ass:locally-stat-hawk-finite-expbounds}.

\noindent Assertions~\ref{item:Prop_tech_part7prim} and~\ref{item:Prop_tech_part8prim} of
Lemma~\ref{lem:Prep_Proof_Thm4} then respectively give
\begin{align}
  \label{eq:1}
  \ellnorm{\mcesi[g]}[2]&\lesssim\ellnorm{g}[2]\\
  \label{eq:2}
  \ellbetanorm{\mcesi[g]}[2][\beta/2]&\lesssim\ellbetanorm{g}[2][\beta/2]+\ellnorm{g}[2]\;.
\end{align}
\paragraph{Norms involving $\mceti[g]$~:}
We apply Lemma~\ref{lem:Prep_Proof_Thm4} with
$f=|g|$, $\varphi(s-t;t)= \locstat{p}(t-s;t/T)$
and $f_\infty = \left|\mcesi[g]\right|$. In
this setting Eq.~(\ref{eq:f_infty_bound}) is inherited from the fact that
$\mceti$ is a fixed point of $\mcet$. Conditions~\ref{item:zeta1-a--lemma} and \ref{item:zetaoo-a--lemma}
 hold with $a=d>0$  as consequences
of~\ref{ass:locally-stat-hawk-finite-expbounds}.
To check~\ref{item:foo_-M--lemma}, we need to choose an appropriate $M<\infty$.
From~\ref{ass:locally-stat-hawk-finite-expbounds}
and~\ref{ass:locally-stat-hawk-unif-lip}, we have
$\overline{\varphi}(u) \leq \locstat{p}(-u;0) + \left(\frac{M}{T}\right)^\beta\
\holderls(-u)$ and thus 
$$
\tilde\zzeta\leq\zetaexpls{0}+ \left(\frac{M}{T}\right)^\beta\
\ellnorm{\holderls}\;.
$$ 
Since $\zetaexpls{0}\leq\zetaexpls{d}<1$ and $\ellnorm{\holderls}<\infty$, we
can define $\varepsilon>0$ small enough, depending only on these two constants
(hence $\varepsilon^{-1}\lesssim1$),
such that, if we set $M=\varepsilon T$ then we have
$\tilde\zzeta\leq\zetaexpls{0}^{1/2}<1$ and \ref{item:foo_-M--lemma} follows.

\noindent Assertions~\ref{item:Prop_tech_part7} and~\ref{item:Prop_tech_part8}
of  Lemma~\ref{lem:Prep_Proof_Thm4} then respectively give
\begin{align}
  \label{eq:3}
    \ellnorm{\mceti[g]}[2]&\lesssim\ellnorm{g}[2]+\rme^{-d \varepsilon T}\ellexpnorm{g}[1][d]\\
  \label{eq:4}
  \ellbetanorm{\mceti[g]}[2][\beta]&\lesssim\ellbetanorm{g}[2][\beta]+
\rme^{-d\varepsilon T/2}  \ellexpnorm{g}[1][d]+    \ellnorm{g}[2]\;.
\end{align}
\paragraph{Norms involving $\mcesi[tilde]$~:}
Applying Lemma~\ref{lem:Prep_Proof_Thm4}~\ref{item:Prop_tech_part2} with $f=|\stat{\tilde{g}}|$
and $f_\infty=\left|\mcesi[tilde]\right|$ and $\phi(u;t)=\stat{p}(-u)$, we get
that $\ellbetanorm{\mcesi[tilde]}\lesssim\ellbetanorm{\stat{\tilde{g}}} +
    \ellnorm{\stat{\tilde{g}}}$.
    \noindent By definition of $\stat{\tilde{g}}$ (adapted
    from~(\ref{eq:tide-g-def}) with $p(s-t;t):=\stat{p}(s-t)$ and $\mcei$ replaced by $\mcesi$), we have
    $\ellnorm{\stat{\tilde{g}}}=\ellnorm{\mcesi}[2]^2\ellnorm{\stat{p}}$ and,
    using~(\ref{eq:weight-trick-vonvolution}),
    $\ellbetanorm{\stat{\tilde{g}}}=\ellbetanorm{\mcesi}[2][\beta/2]^2\ellnorm{\stat{p}}+
    \ellnorm{\mcesi}[2]^2\ellbetanorm{\stat{p}}$.
    By~\ref{ass:locally-stat-hawk-finite-expbounds}, we have
    $\ellnorm{\stat{p}}$, $\ellbetanorm{\stat{p}}\lesssim1$. Hence we finally
    get that
\begin{align}
\label{eq:11}  \ellbetanorm{\mcesi[tilde]}\lesssim\ellnorm{\mcesi}[2]^2+ \ellbetanorm{\mcesi}[2][\beta/2]^2\;.
\end{align}
\paragraph{Norms involving $\mceti[tilde]$~:}
We proceed as in the previous case
and get that
$\ellnorm{\mceti[tilde]}\lesssim\ellnorm{\tilde{g}^{(T)}}$ and
$\ellbetanorm{\mceti[tilde]}\lesssim\ellbetanorm{\tilde{g}^{(T)}} +
\ellnorm{\tilde{g}^{(T)}}$.

\noindent Now $\tilde{g}^{(T)}$ is defined as in~(\ref{eq:tide-g-def}) with
$p(s-t;t):=\locstat{p}(s-t;t/T)$ and $\mcei$ replaced by $\mceti$. We thus have 
$\ellnorm{\tilde{g}^{(T)}}\leq\zetaexpls{0}\ellnorm{\mceti}[2]^2$ and,
using~(\ref{eq:weight-trick-vonvolution}),
$$
\ellbetanorm{\tilde{g}^{(T)}}\leq\zetaexpls{0}\ellbetanorm{\mceti}[2][\beta/2]^2+
\left(\sup_{r}\ellbetanorm{\locstat{p}(\cdot;r)}\right)
\ellnorm{\mceti}[2]^2\;.
$$
    By~\ref{ass:locally-stat-hawk-finite-expbounds}, we have
    $\sup_r\ellbetanorm{\locstat{p}(\cdot;r)}\lesssim1$. Hence we finally
    get that
\begin{align}
\label{eq:13}  \ellnorm{\mceti[tilde]}[1]&\lesssim
  \ellnorm{\mceti}[2]^2\;,  \\
  \label{eq:12}  \ellbetanorm{\mceti[tilde]}[1]&\lesssim
  \ellnorm{\mceti}[2]^2+ \ellbetanorm{\mceti}[2][\beta/2]^2\;.  
\end{align}
\paragraph{Norms involving $\mceti[g]-\mcesi[g]$~:}
Using that $\mceti[g]$ and $\mcesi[g]$ are fixed points of $\mcet$ and $\mces$,
we find that $g_\infty:=\mcesi[g]-\mceti[g]$ satisfies
$$
g_\infty(s) = \int \mceti[g](t)\ \left( \stat{p}(t-s)-\locstat{p}(t-s;t/T)\right)\ \rmd t+
\int g_\infty(t)\ \stat{p}(t-s)\ \rmd t \ \ .
$$
Hence taking absolute values $f_\infty:=|\mcesi[g]-\mceti[g]|$
satisfies~(\ref{eq:f_infty_bound}) 
with
  $$
f(s):=\int |\mceti[g](t)|\ \left|\locstat{p}(t-s;t/T)- \stat{p}(t-s)\right|\ \rmd t\;.
$$
and $\varphi(u;t)=\stat{p}(-u)$. As previously
Conditions~\ref{item:zeta1-a--lemma}, \ref{item:zetaoo-a--lemma}
and~\ref{item:foo_-M--lemma} hold with $a=d>0$ and $M=\infty$ as consequences
of~\ref{ass:locally-stat-hawk-finite-expbounds}. Assertion~\ref{item:Prop_tech_part7prim}
of  Lemma~\ref{lem:Prep_Proof_Thm4} then gives that
$\ellnorm{\mcesi[g]-\mceti[g]}[2]\lesssim\ellnorm{f}[2]$ with $f$ as in the
previous display. By~\ref{ass:locally-stat-hawk-unif-lip}, we further have that
\begin{equation}
  \label{eq:diff-p-stat-locstat}
\left|\locstat{p}(t-s;t/T)- \stat{p}(t-s)\right|\leq T^{-\beta}\,\holderls(t-s)\, \left|t\right|^\beta\;,  
\end{equation}
and thus
\begin{align}\label{eq:6}
\ellnorm{f}[2]=T^{-\beta}\ellnorm{\left(|\mceti[g](\cdot)|\, |\cdot|^\beta\right)\ast
  \holderls}[2]\leq
T^{-\beta}\ellbetanorm{\mceti[g]}[2][\beta] \ellnorm{\holderls}[1]\;.
\end{align}
 Hence, with~(\ref{eq:4}), we finally obtain that
\begin{align}
  \label{eq:5}
  \ellnorm{\mcesi[g]-\mceti[g]}[2]\lesssim
T^{-\beta}\left(  \ellbetanorm{g}[2][\beta]+
\rme^{-d\varepsilon T/2}  \ellexpnorm{g}[1][d]+    \ellnorm{g}[2] \right) \;.
\end{align}
\paragraph{Norms involving $\mceti[tilde]-\mcesi[tilde]$~:}
We apply the same line of reasoning as in the previous case.
Using that $\mceti[tilde]$ and $\mcesi[tilde]$ are fixed point of $\mcet[tilde]$ and $\mces[tilde]$,
we find that $f_\infty:=|\mcesi[\tilde g]-\mceti[\tilde g]|$ satisfies~(\ref{eq:f_infty_bound}) 
with
$$
f(s):= \left|\stat{\tilde{g}}(s)-\tilde{g}^{(T)}(s)\right|+\int  \left|\mceti[\tilde g](t)\right|\ \left| \stat{p}(t-s)-\locstat{p}(t-s;t/T)\right|\ \rmd t\;,
$$
and $\varphi(u;t)=\stat{p}(-u)$. By definition of $\stat{\tilde{g}}$ and $\tilde{g}^{(T)}$ (both adapted
from~(\ref{eq:tide-g-def})), we further have
\begin{align*}
\left|\stat{\tilde{g}}(s)-\tilde{g}^{(T)}(s)\right|
  &\leq
    \int \left|\mceti(t)\right|^2\ \left|\stat{p}(t-s)-\locstat{p}(t-s;t/T)\right|\ \rmd t\\
&\ \ \    +     \int \left|\left(\mcesi(t)\right)^2-\left(\mceti(t)\right)^2\right|\ \stat{p}(t-s)\ \rmd t\; .
\end{align*}
Hence, using~(\ref{eq:diff-p-stat-locstat}), we get that
$$
\ellnorm{f}\leq
T^{-\beta}\,
\left(\ellbetanorm{\mceti}[2][\beta/2]^2+\ellnorm{\mceti[\tilde g]}\right)\ \ellnorm{\holderls}\, 
+
\ellnorm{\left(\mcesi\right)^2-\left(\mceti\right)^2}\ \ellnorm{\stat{p}}\;.
$$
Since $\ellnorm{\stat{p}},\ellnorm{\holderls}\lesssim 1$
under~\ref{ass:locally-stat-hawk-finite-expbounds}
and~\ref{ass:locally-stat-hawk-unif-lip},
Lemma~\ref{lem:Prep_Proof_Thm4}~\ref{item:Prop_tech_part3} with $a=0$ thus yields
\begin{align}
\label{eq:10}\ellnorm{\mcesi[\tilde g]-\mceti[\tilde g]}\lesssim
T^{-\beta}\,
\left(\ellbetanorm{\mceti}[2][\beta/2]^2+\ellbetanorm{\mceti[\tilde g]}\right)\
+ \ellnorm{\left(\mcesi\right)^2-\left(\mceti\right)^2}\;.
\end{align}
\subsection{Conclusion of the proof}
We can now gather the obtained bounds to conclude the proof of
Theorem~\ref{thm:new-var-approx}.

\noindent The bounds~(\ref{eq:7}),~(\ref{eq:13})
and~(\ref{eq:3}) gives~(\ref{eq:bound-of-var-ell2}) (recall that $\varepsilon^{-1}\lesssim1$).

\noindent Finally we prove~(\ref{eq:new-var-approx-hawkes-local}). Using
~(\ref{eq:var-approx-allweightednorms}),
and~(\ref{eq:10}), we first obtain that
\begin{align*}
&\left| \var(N_T(g)) - \var(\stat{N}(g)) \right|
\lesssim T^{-\beta}(\rmI)+\ellnorm{\left(\mceti[g]\right)^2 - \left(\mcesi[g]\right)^2}
\;,\\
\text{with}\qquad&(\rmI):=\ellbetanorm{\mcesi[g]}[2][\beta/2]^2+\ellbetanorm{\mcesi[\tilde g]}+
  \ellbetanorm{\mceti}[2][\beta/2]^2+\ellbetanorm{\mceti[\tilde g]}\;.
\end{align*}
The bounds~(\ref{eq:11}) and~(\ref{eq:12}) and
then~(\ref{eq:1}),~(\ref{eq:2}),~(\ref{eq:3}),~(\ref{eq:4}) (with $\beta/2$
instead of $\beta$) further give
\begin{align*}
(\rmI)&\lesssim
\ellbetanorm{\mcesi[g]}[2][\beta/2]^2+\ellnorm{\mcesi}[2]^2+
        \ellbetanorm{\mceti}[2][\beta/2]^2+\ellnorm{\mceti}[2]^2\\
      &\lesssim \ellnorm{g}[2]^2+\ellbetanorm{g}[2][\beta/2]^2+ \rme^{-A_1 T}  \ellexpnorm{g}[1][d]^2
        \;.
\end{align*}
Using the Hölder inequality and then~(\ref{eq:5}),~(\ref{eq:1})
and~(\ref{eq:3}), we get
\begin{align*}
  \ellnorm{\left(\mceti[g]\right)^2 - \left(\mcesi[g]\right)^2}
  &\leq
    \ellnorm{\mceti[g]-\mcesi[g]}[2]
    \left(\ellnorm{\mceti[g]}[2]+\ellnorm{\mcesi[g]}[2]\right)\\
  &\lesssim
    T^{-\beta}\left(  \ellbetanorm{g}[2][\beta]+
    \rme^{-A_1T}  \ellexpnorm{g}[1][d]+    \ellnorm{g}[2] \right)
        \left(\ellnorm{g}[2]+\rme^{-A_1 T}  \ellexpnorm{g}[1][d]\right)\;.
\end{align*}
Using that the products can be bounded by the sum of squares,
the previous displays yield
\begin{multline*}
  \left| \var(N_T(g)) - \var(\stat{N}(g))
\right|\\\lesssim  T^{-\beta}
\left\{\ellbetanorm{g}[2][\beta/2]^2+ \ellnorm{g}[2]^2+
  \rme^{-A_1 T}\ellexpnorm{g}[1][d]^2
  +\ellbetanorm{g}[2][\beta]\left(\ellnorm{g}[2]+\rme^{-A_1 T}\ellexpnorm{g}[1][d]\right)
\right\}\;.
\end{multline*}
Now, by the Hölder inequality, we have
$\ellbetanorm{g}[2][\beta/2]^2\leq\ellnorm{g}[2]\ellbetanorm{g}[2][\beta]$,
hence the first term inside the curly brackets can be removed by increasing the
multiplicative constant by a factor 2 and we finally
get~(\ref{eq:new-var-approx-hawkes-local}), which concludes the proof of the theorem.

\section{Proof of main results}
\label{sec:proofs}
\subsection{Proof of Theorem~\ref{thm:est_loc_mean_intensity} (local mean density estimation)}
\label{sec:proof-theor-local-mean-est}

For treating the bias, expressed as~(\ref{eq:bias-mean-Nbar-expr}), we
apply~(\ref{eq:mean-approx-hawkes-local}) in
Theorem~\ref{thm:mean-var-approx} with $g=\ktbtime$.
Using the norm estimates of equations~(\ref{eq:W-onenorm}) and~(\ref{eq:W-betanorm}), 
we immediately get
\[
\esp[ \widehat{m}_{\btime}(u_0)] - m_1(u_0) \lesssim T^{-\beta}\ \left(1 + \ellbetanorm{\ktbtime}  \right)\lesssim \btime^\beta + T^{-\beta}\ .
\]
For treating the variance, expressed as~(\ref{eq:var-mea-est-expr-barbar}), we use
Proposition~\ref{prop:burkh-ineq} with $h(\cdot)=\ktbtime(\cdot-Tu_0)$ along
with 
$\ellnorm{h}[\infty] = \ellnorm{\ktbtime}[\infty] =(T\btime)^{-1}$
by~(\ref{eq:W-inftynorm}) and the obvious bound on the
support of the kernel, $\supp(\ktbtime) \lesssim T\btime$. We immediately
get~(\ref{eq:var_loc_mean_dens_est}), which concludes the proof.

\subsection{Proof of Theorem~\ref{thm:total-bias} (Bias of spectral estimator)}
\label{sec:proof-theor-bias-bartlett-est}

The proof of this theorem requires to show two bounds, namely, the bound of the
bias in time direction,~(\ref{eq:total_bias}), and the bound of the
bias in frequency direction,~(\ref{eq:freq_bias}). These two bounds are proved
quite independently. 
\begin{proof}[Proof of~(\ref{eq:total_bias})]
  The derivations of Section~\ref{sec:bias-appr-cent},
  namely~(\ref{eq:est_two_terms_decomp}),~(\ref{eq:est-first-order-square-centered})
  and~(\ref{eq:est-first-order-centered}), show that we can decompose $\esp
  \left[
    \widehat{\gamma}_{\bfreq,\btime}(u_0;\omega_0)-\locstat{\gamma}_{\bfreq}(u_0;\omega_0)\right]$
  as $(\rmI) + (\rmI\rmI) - (\rmI\rmI\rmI)$, where
\begin{align*}
 (\rmI)  & =
\int
\left(\var\left(N_T(S^{-Tu_0}\kbofreq(\cdot-t))\right)-\locstat{\gamma}_{\bfreq}(u_0;\omega_0)\right)\ktbtime(t)\
\rmd t\\
(\rmI\rmI) & =\int \left|\esp[\overline{N}_T(S^{-Tu_0}\kbofreq(\cdot-t);u_0)]\right|^2\ktbtime(t)\ \rmd t \\
(\rmI\rmI\rmI) & =
\underbrace{\var\left(N_T(S^{-Tu_0}\kbofreq\ast\ktbtime)\right)}
+
\underbrace{\left|\esp\left[\overline{N}_T(S^{-Tu_0}\kbofreq\ast\ktbtime;u_0)\right]\right|^2}\\
& \hspace{2cm} (\rmI\rmI\rmI {\rm a})\hspace{5cm} (\rmI\rmI\rmI {\rm b})
\end{align*}
correspond
to~(\ref{eq:bias-bartlett-main-term}),~(\ref{eq:bias-bartlett-neg-term2})
and~(\ref{eq:centering-term-bartlesst-expect-decomp}), respectively.
We will show now that
\begin{enumerate}[label=(\roman*)]
\item\label{item:I} 
the term (I) is of order $\btime^\beta$; 
\item\label{item:II-IIIb} the terms (II) and (IIIb) are of order $\btime^{2\beta} \bfreq^{-1}$;
\item\label{item:IIIa} the term (IIIa) is of order $(T\btime\bfreq)^{-1}\ ; $ 
\end{enumerate}
which will conclude the proof of~(\ref{eq:total_bias}).  

\noindent\textbf{Term (I)}: By~(\ref{eq:new-var-approx-hawkes-local}) in Theorem~\ref{thm:new-var-approx}
with $g=\kbofreq(\cdot-t)$,  and recalling from equations~(\ref{eq:K-twonorm}),
(\ref{eq:K-2betanorm}), (\ref{eq:K-a1norm}), (\ref{eq:betnorm-shifted-bound})
and (\ref{eq:a1norm-shifted-bound})  that $\ellnorm{g}[2] = 1$,
$\ellbetanorm{g}[2]\ \lesssim\ \bfreq^{-\beta} + |t|^\beta$ and
\begin{align*}
\ellexpnorm{g}[1][d]\ \leq\ \rme^{d|t|} \ellexpnorm{\kbofreq}[1][d]\
  \lesssim\ \rme^{d|t|}\ \bfreq^{-1/2} \rme^{A_2\bfreq^{-1}} \lesssim \rme^{d|t|}\
  \rme^{A_2\bfreq^{-1}}  \; ,
\end{align*}
we have for all $t\in\rset$, 
\begin{align*}
\left|\var\left({N_T}(S^{-Tu_0}\kbofreq(\cdot-t))\right) -
  \var\left({N}\left(\kbofreq; u_0\right)\right)\right|
  \lesssim
  T^{-\beta}
  \left(\bfreq^{-\beta} + |t|^\beta+\rme^{A_2|t|}\rme^{A_2\bfreq^{-1}-A_1T}\right)
\end{align*}
where we have used that $1\lesssim\bfreq^{-\beta} \lesssim \rme^{\bfreq^{-1}}$
and that $|t|^\beta \lesssim \rme^{|t|}$.  

\noindent By~(\ref{eq:regularized-new}), we can use this to bound the integrand in the
definition of $(\rmI)$ and thus get, using~(\ref{eq:W-onenorm}),
(\ref{eq:W-betanorm}) and~(\ref{eq:Ktime-a1norm}),
$$
(\rmI)\lesssim
T^{-\beta} \left( \bfreq^{-\beta}  + (T\btime)^\beta +  \rme^{-A_1 T+ A_2 (T\btime+\bfreq^{-1})} \right) \; .
$$
By~(\ref{eq:conditions-T-and-co-timebias-bartlett})
and applying Lemma~\ref{lemma-expo-neglig}, we have that the main term between
the parentheses is the second one, hence we get~\ref{item:I}. 

\noindent\textbf{Term (II)}: applying~(\ref{eq:mean-approx-hawkes-local})  in Theorem~\ref{thm:mean-var-approx}
with $g=\kbofreq(\cdot-t)$ and
using~(\ref{eq:K-betanorm}),~(\ref{eq:K-onenorm})
and~(\ref{eq:betnorm-shifted-bound}), we get 
\begin{align}\nonumber
 \left|\esp[\overline{N}_T(S^{-Tu_0}\kbofreq(\cdot-t);u_0)]\right|
&\lesssim
 T^{-\beta}\left(\bfreq^{-1/2}(1+|t|^\beta)+\bfreq^{-1/2-\beta}\right)\\
&\label{eq:bias-t-u_0-shift}
\lesssim
 T^{-\beta}\left(\bfreq^{-1/2}|t|^\beta+\bfreq^{-1/2-\beta}\right)\;. 
\end{align}
(Since $\bfreq\leq1$)
Taking the square and integrating this with respect to $\ktbtime(t)\rmd t$, and
using~(\ref{eq:W-betanorm}),~(\ref{eq:W-onenorm}) and $T\geq1$, we get
$(\rmI\rmI)\lesssim \btime^{2\beta} \bfreq^{-1}$ as claimed in~\ref{item:II-IIIb}. 

\noindent\textbf{Term (IIIb)}: This term is treated similarly as (II) but using
directly the bounds~(\ref{eq:g-onenorm}) and~(\ref{eq:g-betanorm})
inserted into~(\ref{eq:mean-approx-hawkes-local}) and taking the square. The same order
is obtained and~\ref{item:II-IIIb} follows. 

\noindent\textbf{Term (IIIa)}: Applying the bound~(\ref{eq:bound-of-var-ell2})
of Theorem~\ref{thm:new-var-approx}, we get that, setting here
$g=\kbofreq\ast\ktbtime$, 
$$
\var\left(N_T(g)\right)\lesssim\ellnorm{g}[2]^2+\rme^{- a T}\ellexpnorm{g}[1][d]^2 \;.
$$
Now, from (\ref{eq:g-twonorm}) and (\ref{eq:g-a1norm}), this bound reads
$$
\var\left(N_T(g)\right)\lesssim \bfreq^{-1} (T\btime)^{-1}+\rme^{-A_1 T+ A_2 (T\btime+\bfreq^{-1})} \;.
$$
By Lemma~\ref{lemma-expo-neglig} and since $T\btime\bfreq\geq1$, the first term
dominates and we get~\ref{item:IIIa}, which concludes the proof.
\end{proof}

We now provide a proof for the second part of the theorem controlling the bias.
 \begin{proof}[Proof of~(\ref{eq:freq_bias})] This bound requires the usual control of
   the kernel-regularization of a smooth function as can be seen
   from~(\ref{eq:regularized}). Namely, the function to consider is 
   $\omega\mapsto\locstat{\gamma}(\omega;u_0)$ and the kernel is
$\omega\mapsto \bfreq^{-1}
\left|\fkfreq\left((\omega-\omega_0)/\bfreq\right)\right|^2$ which
integrates to 1 since  $\ellnorm{Q}[2]=\ellnorm{q}[2]/\sqrt{2\pi}=1$
by~\ref{ass:kfreq-comp-support} and the Parseval theorem.
Using the formula~(\ref{eq:local-bartlett-spectral-density}) to
express $\omega\mapsto\locstat{\gamma}(\omega;u_0)$ and the usual conditions on
the kernel~(\ref{eq:addtional-cond-Kkernek}), it is thus sufficient to prove
that, for all $\omega\in\rset$,
$$
\frac{\locstat{m}_1(u_0)}{2\pi}\left||1-\locstat{\fourierp}(\omega;u_0)|^{-2}-|1-\locstat{\fourierp}(\omega_0;u_0)|^{-2}+
C\,(\omega-\omega_0)\right| \lesssim (\omega-\omega_0)^2\;,
$$
where $C$ is any constant (possibly depending on $\omega_0$ and $u_0$ but not
on depending on $\omega$). 
As already seen, we have 
$$
\locstat{m}_1(u_0)\leq\ellnorm{\locstat{\lambda}_c}[\infty]/(1-\zetals)\lesssim1\;,
$$
so that, we can consider the ratio $\locstat{m}_1(u_0)/(2\pi)$ as a constant. 
For the remaining term involving the function $\omega\mapsto|1-\locstat{\fourierp}(\omega;u_0)|^{-2}$,
we first observe that
\begin{multline*}
 \locstat{\fourierp}(\omega;u_0)-\locstat{\fourierp}(\omega_0;u_0) \  = \int \locstat{p}(t;u_0) (\rme^{-\rmi\omega t }- \rme^{-\rmi\omega_0 t } )\ \rmd t\\
= \tilde{p}(\omega_0;u_0)\ (\omega-\omega_0) + \int \locstat{p}(t;u_0) \rme^{-\rmi\omega_0 t }\ \left(\rme^{-\rmi(\omega-\omega_0) t }- 1 - \rmi(\omega-\omega_0) t  \right)\  \rmd t\ ,
\end{multline*}
with $\tilde{p}(\omega_0;u_0) := \int \rmi\ t\ \locstat{p}(t,u_0)\ \rme^{-\rmi\omega_0 t }\ \rmd t\ .$
In the latter display, the first term is of the form $C(\omega-\omega_0)$ with
$|C|\lesssim 1$ and the second term is of order $(\omega-\omega_0)^2$. This
comes respectively from
$$
\left|\tilde{p}(\omega_0;u_0)\right|\leq 
\int |t|\ \locstat{p}(t;u_0) \rmd t\leq \zetals+\zetabet[2]
$$
and
$$
\left|
\int \locstat{p}(t,u_0) \rme^{-\rmi\omega_0 t }\ \left(\rme^{-\rmi(\omega-\omega_0) t }- 1 - \rmi(\omega-\omega_0) t  \right)\  \rmd t
\right|\leq \zetabet[2] \;(\omega-\omega_0)^2 \;. 
$$
To conclude the proof we argue that the form $C(\omega-\omega_0)+R(\omega)$ with
$C\lesssim1$ and $R(\omega)\lesssim(\omega-\omega_0)^2$ satisfied by  $\locstat{\fourierp}(\omega;u_0)-\locstat{\fourierp}(\omega_0;u_0)$
is inherited by 
 $|1-\locstat{\fourierp}(\omega;u_0)|^{-2}-|1-\locstat{\fourierp}(\omega_0;u_0)|^{-2}$.
This follows from 
$$
|1-\locstat{\fourierp}(\omega;u_0)|^{-1}\leq (1-\zetals)^{-1}\in(1,\infty)
$$
(using~(\ref{eq:locally-stat-hawk-finite-intensity})) and from the 
identity valid for all complex numbers $z,z'$ inside the unit disk
\[
\frac1{|1-z|^2}-\frac1{|1-z'|^2}\ =\ \frac{|1-z|^2-|1-z'|^2}{|1-z'|^4}\ +\ \frac{\left(|1-z|^2-|1-z'|^2\right)^2}{|1-z'|^4\ |1-z|^2}
\]
The numerator of the first term of this sum can be expressed as a sum of terms depending on the difference $|z-z'|$, and this applies to the numerator of the second term, too (as it is the square of the first numerator):
\[
|1-z|^2-|1-z'|^2\ =\ 2\Re(z-z')\Re(1-z') + 2\Im(z-z')\Im(1-z') \ +\ |z-z'|^2\;.
\]
\end{proof}
\subsection{Additional lemmas}
As explained in Section~\ref{sec:variance-terms-exact}, the treatment of
the variance is done via the introduction of the ``truly'' centered process
$\overline{\overline{N}}_T$. Here we provide three lemmas, two concerned with
useful bounds for this centered process and the third one which controls the
quality of the approximation of the estimator based on the centered process. 
\begin{lemma}
  \label{lem:useful-bounds-exact-centering}
  Let $p\geq1$.  Under the conditions of
  Theorem~\ref{thm:variance_loc_Bartlett_est}, we have, for all $u_0 \in
  \Rset$, $\omega_0 \in \Rset$, and $\btime,\bfreq,T$ as
  in~(\ref{eq:conditions-T-and-co-timebias-bartlett}),
  \begin{align}
\label{eq:varfstarw}
\left\|\overline{\overline{N}}_T(S^{-Tu_0}\kbofreq\ast \ktbtime)\right\|_p&\lesssim 
\left(T\btime\bfreq\right)^{-1/2}\;,
\end{align}
where $\overline{\overline{N}}_T$ is defined in~(\ref{eq:true-centering-def}).
Let moreover $h:\rset\to\cset$ be such that, for all $t\in\rset$, we
have
$$
|h(t)|\leq \left(a_T+b_T|t|^{\beta}\right)\ktbtime(t)\;,
$$
for two positive constants $a_T$ and $b_T$ (possibly depending on   $T$,
$\btime$ and $\bfreq$). Then we have, for all $u_0 \in
  \Rset$, $\omega_0 \in \Rset$, and $\btime,\bfreq,T$ as
  in~(\ref{eq:conditions-T-and-co-timebias-bartlett}),
\begin{align}
\label{eq:varfstarhw}
\left\|\overline{\overline{N}}_T(S^{-Tu_0}\kbofreq\ast h)\right\|_p&\lesssim 
(a_T+b_T(\btime T)^{\beta})
\left(T\btime\bfreq\right)^{-1/2}\;.
      \end{align}
\end{lemma}
\begin{proof}
We apply Proposition~\ref{prop:burkh-ineq}, and get
$$
\left\|\overline{\overline{N}}_T(S^{-Tu_0}\kbofreq\ast \ktbtime)\right\|_p \leq A \;\ellnorm{\kbofreq\ast \ktbtime}[\infty]
\sqrt{n}\;,
$$
for some generic constant $A$ and with a positive integer upper bound  $n$ on
the length of the support of $S^{-Tu_0}\kbofreq\ast \ktbtime$, denoted by $\supp(S^{-Tu_0}\kbofreq\ast \ktbtime)$.
Observing that $\supp(S^{-Tu_0}\kbofreq\ast \ktbtime) \subset \supp(S^{-Tu_0}\kbofreq) + \supp(\ktbtime)$, that
the length of $\supp(\kbofreq)$ is of order $\bfreq^{-1}$ and that
the length of $\supp(\ktbtime)$ is of order $T\btime$, we have
$n\lesssim \bfreq^{-1}+T\btime\lesssim T\btime$. We thus obtain the
bound~(\ref{eq:varfstarw}) with~(\ref{eq:g-inftynorm}).

The bound~(\ref{eq:varfstarhw}) is obtained similarly but this time we rely on the
bound 
$$
\ellnorm{\kbofreq\ast h}[\infty]\leq a_T \ellnorm{\kbofreq\ast
  \ktbtime}[\infty]+
b_T\ellnorm{\kbofreq}\ellnorm{\ktbtime}[\infty]\,n^\beta\;.
$$
(Recall that $n$ is length of the support of $\ktbtime$.)
With~(\ref{eq:g-inftynorm}), $n\lesssim T\btime$,~(\ref{eq:K-inftynorm}) and~(\ref{eq:W-betanorm})
we get
$\ellnorm{\kbofreq\ast h}[\infty]\lesssim 
\bfreq^{-1/2}(\btime T)^{-1}(a_T+b_T
(\btime T)^{\beta})$, which yields~(\ref{eq:varfstarhw}).
\end{proof}
\begin{lemma}
  \label{lem:var_lead_term}
  Under the conditions of Theorem~\ref{thm:variance_loc_Bartlett_est}, we have, for all $u_0 \in
  \Rset$, $\omega_0 \in \Rset$, and $\btime,\bfreq,T$ as
  in~(\ref{eq:conditions-T-and-co-timebias-bartlett}),
  \begin{equation}
    \label{eq:var-gammatilde-bound}
\var\left(\int |\overline{\overline{N}}_T(S^{-Tu_0}\kbofreq(\cdot-t))|^2 \ktbtime(t) \;\rmd t
\right)\ \lesssim\ (T\btime\bfreq)^{-1} \;.    
  \end{equation}
\end{lemma}
\begin{proof}
We can write the left-hand side of~(\ref{eq:var-gammatilde-bound}) as
$$
\iint \cov\left(|\overline{\overline{N}}_T(f(\cdot-t))|^2 , |\overline{\overline{N}}_T(f(\cdot-t'))|^2 \right)\ \ktbtime(t)\ \ktbtime(t')\ \;\rmd t \;\rmd t' \ .
$$
Let $\ell$ denote the length of the support of $\kbofreq$, which clearly
satisfies for $\bfreq\in(0,1]$,
\begin{equation}
  \label{eq:support-bound-kbofreq}
  \ell\lesssim \bfreq^{-1} \;.
\end{equation}
 Then with
 $h_1=\kbofreq(\cdot-t)$ and $h_2=\kbofreq(\cdot-t')$, setting $\gamma:=(|t-t'|
 - \ell)_+$, we have one of the
 assertions~\ref{item:gamma-ass1}, \ref{item:gamma-ass2}
 or~\ref{item:gamma-ass3} which is satisfied. Hence
 Corollary~\ref{cor:cov-squares} gives that, for some $q>4$,
\begin{equation}\label{eq:gap_cov_bound}
\left| \cov\left(\left|\overline{\overline{N}}_T(h_1)\right|^2,\left|\overline{\overline{N}}_T(h_2)\right|^2\right)\right|
\leq C_q  \left\|\overline{\overline{N}}_T(h_1)\right\|_q^2\;
\left\|\overline{\overline{N}}_T(h_2)\right\|_q^2
\;\rme^{-\alpha_q (|t-t'| - \ell)_+ } \; .
\end{equation}
Further we apply Proposition~\ref{prop:burkh-ineq} with~(\ref{eq:K-inftynorm})
and~(\ref{eq:support-bound-kbofreq}) and get, for $i=1,2$,
$\left\|\overline{\overline{N}}_T(h_i)\right\|_q^2  \lesssim 
\left(\ellnorm{\kbofreq}[\infty]\right)^2 \ell\lesssim1 \;.$
Hence we finally get
\begin{align*}
\var\left(\int |\overline{\overline{N}}_T(S^{-Tu_0}\kbofreq(\cdot-t))|^2 \ktbtime(t) \;\rmd t
\right)\ &\lesssim\iint \rme^{-\alpha_q (|t-t'| - \ell)_+ } \ktbtime(t)\
\ktbtime(t')\ \;\rmd t \;\rmd t'\\
&\leq\ellnorm{\ktbtime\ast\ktbtime}[\infty]\int \rme^{-\alpha_q (|u| - \ell)_+ }\,\rmd u\; .
\end{align*}
Now, by~(\ref{eq:support-bound-kbofreq}) we have $\int \rme^{-\alpha_q (|u| -
  \ell)_+ }\,\rmd u\lesssim \bfreq^{-1}$ and by~(\ref{eq:W-onenorm})
and~(\ref{eq:W-inftynorm}), we have
$\ellnorm{\ktbtime\ast\ktbtime}[\infty]\lesssim (\btime T)^{-1}$. Hence we
get~(\ref{eq:var-gammatilde-bound}) and the proof is concluded.
\end{proof}
\begin{lemma}
  \label{lem:approximation-exact-centering}
  Under the conditions of Theorem~\ref{thm:variance_loc_Bartlett_est}, we have,
 for all $u_0 \in \Rset$, $\omega_0 \in \Rset$, and $\btime,\bfreq,T$ as
  in~(\ref{eq:conditions-T-and-co-timebias-bartlett}),
  \begin{equation}
    \label{eq:gamma-tildegamma-approx}
\left\|\widehat{\gamma}_{\bfreq,\btime}(u_0;\omega_0)-\widetilde{\gamma}_{\bfreq,\btime}(u_0;\omega_0)\right\|_2
\lesssim \btime^{2\beta} \bfreq^{-1} +  \left(T\btime\bfreq\right)^{-1}
\;,    
  \end{equation}
where $\widehat{\gamma}_{\bfreq,\btime}(u_0;\omega_0)$ and
$\widetilde{\gamma}_{\bfreq,\btime}(u_0;\omega_0)$ are respectively defined
by~(\ref{eq:kernel-est}) and~(\ref{eq:exact-centering-approx}). 
\end{lemma}
\begin{proof}
By definitions~(\ref{eq:approx-centering})
and~(\ref{eq:true-centering-def}), we have, for any integrable test function $f$,
$\overline{\overline{N}}_T(f)=\overline{N}_T(f;u_0)-\esp[\overline{N}_T(f;u_0)]$.
Thus,~(\ref{eq:est_two_terms_decomp}) and~(\ref{eq:est-first-order-centered})
yield the following expression for $\widehat{\gamma}_{\bfreq,\btime}(u_0;\omega_0)$
$$
\int\left|\overline{\overline{N}}_T(f(\cdot-t))+\esp[\overline{N}_T(f(\cdot-t);u_0)]\right|^2\,\ktbtime(t)\;\rmd
 t - \left|\overline{\overline{N}}_T(f\ast\ktbtime)+\esp[\overline{N}_T(f\ast\ktbtime;u_0)] \right|^2 \ ,
$$
where we used the test function $f=S^{-Tu_0}\kbofreq$.
Developing the first square modulus and using for the second that
$|a+b|^2\leq2|a|^2+2|b|^2$, and since  $\int\ktbtime=1$, we obtain
\begin{equation}
  \label{eq:decomp1-approx-bartlett-centered}
\left|\widehat{\gamma}_{\bfreq,\btime}(u_0;\omega_0)-\widetilde{\gamma}_{\bfreq,\btime}(u_0;\omega_0)-
2\Re (B_T)\right|\leq
A_T+2C_T+2D_T\;,  
\end{equation}
where we set, denoting by $\conjugate{f}$ the complex conjugate of $f$, 
\begin{align*}
  A_T&:=
  \int\left|\esp[\overline{N}_T(f(\cdot-t);u_0)]\right|^2\,\ktbtime(t)\;\rmd t \;,
  \\
B_T&:=\int \overline{\overline{N}}_T(f(\cdot-t))\esp[\overline{N}_T(\conjugate{f}(\cdot-t);u_0)]\,\ktbtime(t)\;\rmd
 t\;,\\
 C_T&:=\left|\overline{\overline{N}}_T(f\ast\ktbtime)
 \right|^2\quad\text{and}\quad 
D_T:=\left|\esp[\overline{N}_T(f\ast\ktbtime;u_0)]\right|^2\;.
\end{align*}
Note that $A_T$ and $D_T$ have been treated in the proof of
Theorem~\ref{thm:total-bias} as  the (deterministic) terms
$(\rmI\rmI)$ and $(\rmI\rmI\rmI b)$.
The assumptions in Theorem~\ref{thm:total-bias} are weaker than that of this
lemma. Hence we can directly use~\ref{item:II-IIIb} of the proof of Theorem~\ref{thm:total-bias} and obtain that
$$
A_T,D_T\lesssim\btime^{2\beta} \bfreq^{-1}\;.
$$
The bound~(\ref{eq:varfstarw}) in Lemma~\ref{lem:useful-bounds-exact-centering} immediately gives
$$
\|C_T\|_2=\left\|\overline{\overline{N}}_T(S^{-Tu_0}\kbofreq\ast\ktbtime)\right\|_4^2\lesssim
\left(T\btime\bfreq\right)^{-1} \;.
$$
and we are left with treating $B_T$. 
Note that we can rewrite $B_T$ as
$B_T= \overline{\overline{N}}_T(f\ast h)$, where $h$ is the (deterministic) function $t\mapsto
\esp[\overline{N}_T(\conjugate{f}(\cdot-t);u_0)]\,\ktbtime(t)$. Now,
by~(\ref{eq:bias-t-u_0-shift}), we have, for all $t\in\rset$,
$$
|h(t)|\lesssim
\bfreq^{-1/2} T^{-\beta}\left(\bfreq^{-\beta}+|t|^\beta\right) \ktbtime(t)\;.
$$
Applying~(\ref{eq:varfstarhw}) with $a_T=\bfreq^{-1/2} (\bfreq T)^{-\beta}$ and $b_T=\bfreq^{-1/2} T^{-\beta}$
in Lemma~\ref{lem:useful-bounds-exact-centering}, it
follows that 
$$
\|B_T\|_2\lesssim
\bfreq^{-1/2}T^{-\beta}\left(\bfreq^{-\beta}+(\btime T)^\beta\right)(T\btime\bfreq)^{-1/2}
\lesssim \bfreq^{-1/2}\btime^\beta(T\btime\bfreq)^{-1/2}\;,
$$
for $T\btime\bfreq\geq1$.

Inserting the previous bounds on $A_T$, $B_T$, $C_T$ and $D_T$
in~(\ref{eq:decomp1-approx-bartlett-centered}) yields
$$
\left\|\widehat{\gamma}_{\bfreq,\btime}(u_0;\omega_0)-\widetilde{\gamma}_{\bfreq,\btime}(u_0;\omega_0)\right\|_2
\lesssim \bfreq^{-1/2}\btime^\beta(T\btime\bfreq)^{-1/2}+\btime^{2\beta}
\bfreq^{-1} +  (T\btime\bfreq)^{-1}
\;.
$$
Using that $2\bfreq^{-1/2}\btime^\beta(T\btime\bfreq)^{-1/2}\leq\btime^{2\beta}
\bfreq^{-1} +  (T\btime\bfreq)^{-1}$, we get~(\ref{eq:gamma-tildegamma-approx}). 
\end{proof}
\subsection{Proof of Theorem~\ref{thm:variance_loc_Bartlett_est} (variance of spectral estimator)}

Lemmas~\ref{lem:var_lead_term} and~\ref{lem:approximation-exact-centering}, together with
the definition of $\widetilde{\gamma}_{\bfreq,\btime}(u_0;\omega_0)$
in~(\ref{eq:exact-centering-approx}), directly give that 
$$
\var\left(\widehat{\gamma}_{\bfreq,\btime}(u_0;\omega_0)\right) \ \lesssim\  \
\frac{1}{T\btime\bfreq} + \left( \btime^{2\beta} \bfreq^{-1} +  (T\btime\bfreq)^{-1}\right)^2\;.
$$
Since
$T\btime\bfreq\geq1$ in~(\ref{eq:conditions-T-and-co-timebias-bartlett}), the
second term within the squared parentheses is at most of the same order as the
term outside the squared parentheses and can thus be discarded.
Hence we obtain Theorem~\ref{thm:variance_loc_Bartlett_est}.

\section{Proof of Proposition~\ref{prop:uniform-moment-bound}}
\label{sec:proof-prop-moment-bound}
This proof uses the notation and definitions of
\cite[Section~2.1]{roueff-vonsachs-sansonnet2016}, the essential of which we
now briefly recall.  Let $m$ be a positive integer and $\ouvert$ be an open
subset of $\cset^m$.  Define $\dd$ as the set of holomorphic functions from
$\ouvert$ to $\rset$. We denote, for all $h\in\dd$ and compact sets
$K\subset\ouvert$,
$$
\normd{h}[K]=\sup_{\zed\in K}\left| h(\zed)\right| \;.
$$
Recall that a holomorphic function $h$ on $\ouvert$ is infinitely
differentiable on $\ouvert$.
We denote by $\D$ the set of $\Rset\times\ouvert\to\Rset$ functions $h$
such that, for all $t\in\Rset$, $\zed\mapsto h(t,\zed)$ belongs to
$\dd$.  
For any $p\in[1,\infty]$, we further denote by $\Dp[p]$ the subset of functions
$h\in\D$ such that the function $t\mapsto \sup_{\zed\in K}{h(t,\zed)}$ has
finite $L^p$-norm on $\rset$ for all compact sets $K\subset\ouvert$.
We denote
$$
\normD{h}[p]:=
\ellnorm{ \sup_{\zed\in K}\left|h(\cdot,\zed)\right|}[p] \;.
$$
We also denote by $\bdp{r}[p]$ the set of all functions $g\in \Dp[p]$ such that
$\normD{g}[p]<r$.

Consider now any
\begin{align}\label{eq:def-r1-rinf-D}
r_\infty\in(0,-\log\zzeta)
\quad\text{and}\quad
r_1\in \left(0,r_{\infty}\rme^{-r_\infty}\zzeta[\infty]^{-1}\right)\;.
\end{align}
and set
\begin{align}
\label{eq:cond-gD-1}
& R_1:= r_1\left(1-\zzeta\rme^{r_\infty}\right) \in(0,r_1) \;,\\
\label{eq:cond-gD-infty}
& R_\infty:= r_\infty-\rme^{r_\infty}\zzeta[\infty]r_1\in(0,r_\infty)\;.
\end{align}
Let $K\subset\ouvert$ be a compact set and
$g\in\bdp{R_1}\cap\bdp{R_\infty}[\infty]$.

Corollary~12 and Eq.~(33) in~\cite{roueff-vonsachs-sansonnet2016} give that if
  $g\in\bdp{R_1}\cap\bdp{R_\infty}[\infty]$, with
 $R_1$, $R_\infty$ defined
  by~(\ref{eq:cond-gD-1}) and~(\ref{eq:cond-gD-infty}), we have, for all
  $\zed\in K$,
$$
\mathcal{L}(g(\cdot,\zed)):=\esp\left[\rme^{N(g(\cdot,\zed))}\right] = \exp\int
\left(\exp\left(\Phi_{g}^\infty(t^c,\zed)\right)-1\right)\,
\lambda_c(t^c)\,\rmd t^c\;,
$$
where $\Phi_{g}^\infty$ is defined
in~\cite[Definition~3]{roueff-vonsachs-sansonnet2016} as as element of
$\bdp{r_1}\cap\bdp{r_\infty}[\infty]$.

Let now $h:\rset\to\rset_+$ be a bounded and integrable function and set
$g(t,\zed)=\zed\,h(t)$. Let $\ouvert=\rset$ and $K=[-r,r]$ for some
$r>0$. The previous display and the bound $|\rme^a-1|\leq
|a|\rme^{|a|}$ give that
$$
\sup_{|\zed|\leq r}\left|\esp\left[\rme^{\zed\,N(h)}\right]\right|
\leq \exp\left(\ellnorm{\lambda_c}[\infty]
\,\rme^{\normD{\Phi_{g}^\infty}[\infty]}
\,\normD{\Phi_{g}^\infty}
\right)\leq
\exp\left(\ellnorm{\lambda_c}[\infty]
\,\rme^{r_{\infty}}\,r_1
\right)\;.
$$
Here $r_1$ and $r_\infty$ should be taken as small as possible provided
that~(\ref{eq:def-r1-rinf-D}) holds and   $g\in\bdp{R_1}\cap\bdp{R_\infty}[\infty]$, with
 $R_1$, $R_\infty$ defined
  by~(\ref{eq:cond-gD-1}) and~(\ref{eq:cond-gD-infty}).
The specific choice of $g$ and $K$ here gives
$g\in\bdp{R_1}\cap\bdp{R_\infty}[\infty]$ if
$$
r\ellnorm{h}\leq R_1\quad\text{and}\quad
r\ellnorm{h}[\infty]\leq R_\infty\;.
$$
So we conclude that
$$
\sup_{|\zed|\leq r}\left|\esp\left[\rme^{\zed\,N(h)}\right]\right|
\leq \exp\left(\ellnorm{\lambda_c}[\infty]
\,\rme^{r_{\infty}}\,r_1
\right)\;,
$$
for any  $r_1$ and $r_\infty$ satisfying~(\ref{eq:def-r1-rinf-D}) and $r$ satisfying
$$
r< \min\left(\frac{r_1\left(1-\zzeta\rme^{r_\infty}\right)}{\ellnorm{h}},
\frac{r_\infty-\rme^{r_\infty}\zzeta[\infty]r_1}{\ellnorm{h}[\infty]}\right)\;.
$$
Let us set $r_\infty=(-\log\zzeta)/2$ so that~(\ref{eq:def-r1-rinf-D}) reduces to
$$
0<r_1<(-\log\zzeta)\zzeta^{1/2}\zzeta[\infty]^{-1}/2\;.
$$
In the  particular case where $\ellnorm{h}\leq1$ and
$\ellnorm{h}[\infty]\leq1$, the condition on $r$ then reads
as
$$
r<\min(r_1(1-\zzeta^{1/2}),(-\log\zzeta)/2-r_1\zzeta[\infty]\zzeta^{-1/2})
=r_1(1-\zzeta^{1/2})\;,
$$
where the last equality holds for the choice of $r_1$ given
by~(\ref{eq:r1def-exp-bound}). We thus get, for all $r<r_1(1-\zzeta^{1/2})$,
$$
\esp\left[\rme^{r\,N(h)}\right] \leq
\exp\left(\ellnorm{\lambda_c}[\infty]
\,\zzeta^{-1/2}\,r_1
\right)\ .
$$
Letting $r$ tend to $r_1(1-\zzeta^{1/2})$, we also get the result for
$r=r_1(1-\zzeta^{1/2})$ which corresponds to~(\ref{eq:exp-bound}) for a
non-negative $h$. To conclude, if $h$ is signed we use that $|N(h)|\leq N(|h|)$
and apply the previous bound to $|h|$.

\end{document}